\documentclass{article}
\usepackage[utf8]{inputenc}
\usepackage{eurosym}
\usepackage{amstext,amsthm}
\usepackage{amsmath, a4}
\usepackage{amssymb,mathtools}
\usepackage[nointegrals]{wasysym} % enthaelt iint-Def die aber auch in amsmath definiert sind
\usepackage{mathrsfs}
\usepackage{ulem}
\usepackage{graphicx}
\usepackage{enumerate}

\usepackage{dsfont,bbm,mathtools,mathrsfs,color,wasysym,caption,pdfsync, pgfplots, tkz-euclide}
\usepackage{tikz, pictexwd}
\usetkzobj{all}
\usepackage[sort]{natbib}

\newtheorem{proposition}{Proposition}[section]
\newtheorem{corollary}[proposition]{Corollary}
\newtheorem{apptheorem}[proposition]{Theorem}
\newtheorem{theorem}{Theorem}
\newtheorem{lemma}[proposition]{Lemma}
\theoremstyle{definition}
\newtheorem{definition}[proposition]{Definition}
\newtheorem{remark}[proposition]{Remark}
\DeclareMathAlphabet{\mathpzc}{OT1}{pzc}{m}{it}
\numberwithin{equation}{section}
\numberwithin{figure}{section}

\newcommand\unnumberedfootnote[1]{ %
        \let\temp=\thefootnote %
        \renewcommand{\thefootnote}{}%
        \footnote{#1}%
        \let\thefootnote=\temp%
        \addtocounter{footnote}{-1}}

\makeatletter
\renewcommand{\@fnsymbol}[1]{\@arabic{#1 }}
\makeatother

\begin{document}

\title{\LARGE The fixation probability and time for a doubly
  beneficial mutant}

\author{\sc S. Bossert\footnote{Abteilung Mathematische Stochastik,
    Universit\"at Freiburg, D - 79104 Freiburg},
  P. Pfaffelhuber$\mbox{}^{1,}$\footnote{Corresponding author; email
    p.p@stochastik.uni-freiburg.de}}

\date{\today}

\maketitle

\begin{abstract}
  For a highly beneficial mutant $A$ entering a randomly reproducing
  population of constant size, we study the situation when a second
  beneficial mutant $B$ arises before $A$ has fixed. If the selection
  coefficient of $B$ is greater than the selection coefficient of $A$,
  and if $A$ and $B$ can recombine at some rate $\rho$, there is a
  chance that the double beneficial mutant $AB$ forms and eventually
  fixes. We give a convergence result for the fixation probability of
  $AB$ and its fixation time for large selection coefficients.
\end{abstract}

\unnumberedfootnote{Keywords: Interacting Wright--Fisher diffusions;
  ancestral selection recombination graph; branching process
  approximation; inhomogeneous branching process with
  immigration} % Separate items with ;

\unnumberedfootnote{AMS Subject Classification: 92D15 (Primary);
  60J80, 60J85, 92D10 (Secondary)}

\section{Introduction}
The spread of a beneficial mutant in a constant size population is a
well-studied model in mathematical population genetics. A first
approximation of its fixation probability has already been established
by \cite{haldane1927}, and the theory of one-dimensional diffusions
can be used in order to obtain properties of the fixation time
\citep{KimuraOhta1969}; see also \cite{EtheridgeEtAl2006} and
\cite{HermissonPennings2005}. Adding recombination with a neutral
locus rises new questions about the genealogy at the neutral locus as
well as an opportunity to detect selection from a population sample
(see e.g.\
\citealp{MaynardSmithHaigh1974,Barton1998,StephanWieheLenz1992,KaplanHudsonLangley1989,sweepfinder}).

However, population genetic models become much more complex if we
assume that a second (different) beneficial mutant arises during the
spread of the first. Due to recombination, there is a chance that the
first and the second beneficial mutant recombine to form an even
fitter type. While \cite{OttoBart1997} and \cite{Barton1998} have
studied the case that the second allele is less beneficial than the
first, we will concentrate on the opposite case. This is even more
interesting since only a more beneficial second mutant has a chance to
survive against an almost fixed first mutant and form the fittest
recombinant type. This model of competing selective sweeps has been
studied in a series of papers \citep{stephan1995, KirbyStephan1996,
  chevinetal2008, YuEth2012, HartfieldOtto2011, Cuthetal2012}, but the
picture is not complete yet.

Basic questions are the fixation probability, fixation time and the
pattern of this scenario in genetic data under the competing sweeps
model. In the present paper, we are dealing with the first two
questions using a diffusion model. This complements previous work by
\cite{YuEth2012} and \cite{Cuthetal2012}, who use a Moran model and
studied the fixation probability of the fittest recombinant type. We
are able to extend their results in several respects: (i) the
probability of fixation is given explicitly (in the limit of large
selection coefficients); (ii) conditional on fixation, we obtain an
approximation of the fixation time of the fittest recombinant type.

The main method we use is based on the ancestral selection graph
\citep{KroNeu1997,NeuKrone1997}, which describes possible ancestral
lines in selective population genetic models. Recently, this graph has
been used to study the process of fixation for selective sweeps
\citep{PokPfa2012,greven2014fixation}. Since we are dealing with two
recombining loci, we have to follow \cite{GrifMarj1997} and add
recombination events to this graph in order to obtain the ancestral
selection recombination graph. This approach is not new and is
implicit in the ancestral influence graph of
\cite{DonnellyKurtz1999}{\color{black}; see also \cite{Fearnhead2002}}.
\cite{LessardKermany2012} combined selection and recombination in one
graph for a fixed population size and a Markov chain in discrete time
and \cite{mano2008} analysed the evolution of duplicated genes under
the influence of selection and recombination. However, the limit of
large selection coefficients including recombination has not been
studied using the ancestral selection {\color{black} recombination}
graph before.

After introducing the diffusion model in Section~\ref{S:model}, we
give our main results. In Section~\ref{S:ASRG}, we explain
our main technique, the ancestral selection recombination graph (ASRG)
and translate the event of fixation within the diffusion model to
properties of the ASRG in terms of a Markov jump process. Before we
come to the proofs of the main results, we give auxiliary results in
Section~\ref{S:aux}. Then, in Section~\ref{S:proofs}, we are ready to
give the proofs of our Theorems.

\section{Model and main result}
\label{S:model}
\subsection{Competing sweeps}
We use the standard diffusion model from population genetics including
selection and recombination (see e.g.\ \citealp{OhtaKimura1969b,
  EthierKurtz1993, Ewens2004}) with the four types
\begin{align*}
  0\equiv ab, \qquad 1 \equiv Ab, \qquad 2\equiv aB, \qquad 3\equiv AB,
\end{align*}
which have selection coefficients
$0=\alpha_0 < \alpha_1 < \alpha_2 < \alpha_3 \equiv \alpha$. The
evolution of the frequencies of these types is governed by the
solution of the system of SDEs
\begin{equation}\label{eq:diffsystem} 
  \begin{aligned}
    dX_0 & = \Big(- X_0 \sum_j \alpha_j X_j + \rho(X_1X_2 -
    X_0X_3)\Big)dt +
    \sum_{j\neq 0} \sqrt{X_0X_j}dW_{0j},\\
    dX_1 & = \Big(X_1\Big(\alpha_1 - \sum_j \alpha_j X_j\Big) +
    \rho(X_0X_3 - X_1X_2)\Big)dt
    + \sum_{j\neq 1} \sqrt{X_1X_j}dW_{1j},\\
    dX_2 & = \Big(X_2\Big(\alpha_2 - \sum_j \alpha_j X_j\Big) +
    \rho(X_0X_3 - X_1X_2)\Big)dt
    + \sum_{j\neq 2} \sqrt{X_2X_j}dW_{2j},\\
    dX_3 & = \Big(X_3\Big(\alpha_3 - \sum_j \alpha_j X_j\Big) +
    \rho(X_1X_2 - X_0X_3)\Big)dt + \sum_{j\neq 3}
    \sqrt{X_3X_j}dW_{3j},
  \end{aligned}
\end{equation}
where $(W_{kl})_{k>l}$ is a family of independent Brownian motions and
$W_{lk} =-W_{kl}$. Note that $X_{0}+X_{1}+X_{2}+X_{3}=1$ for all
times, if the initial state satisfies this relation. Here, $X_i(t)$
denotes the frequency of type $i$ at time $t$. We will write
$\mathcal X = (\underline X(t))_{t\geq 0}$ with
$\underline X = (X_0, X_1, X_2, X_3)$ for a solution of
~\eqref{eq:diffsystem}, whose existence and uniqueness follows from
standard theory \citep{EthierKurtz1993}.

\subsection{Main results}
We now give our main results on the fixation probability and fixation
time (conditioned on fixation). Their proofs are given in
Section~\ref{S:proofs}.

\begin{theorem}[Fixation probability of type $3\equiv AB$]\label{theoremfixprob2}
  Let {\color{black}$c, \psi >0$ and $0<\delta< 1$} and
  $\mathbb P_{\underline{x}_{\delta,\psi}}$ be the distribution of the
  solution $\mathcal{X}$ of \eqref{eq:diffsystem}, started in
  $\underline{x}_{\delta,\psi} = (1-\delta-c\alpha^{-\psi},
  c\alpha^{-\psi},\delta,0)$. Assume that
  \begin{equation*}
    \frac{\alpha_{i}}{\alpha}  \xrightarrow{\alpha \to \infty} c_{i},
  \end{equation*} 
  with $0 < c_{1} < c_{2} < c_{3}=1$. 
  \begin{enumerate}
  \item If $\psi < \frac{c_{1}}{c_{2}}$, the fixation probability of
    $3\equiv AB$ satisfies
    \begin{equation}\label{finalfixresult}
      \lim_{\alpha \to \infty} 
      \frac{1}{{\color{black}2}\alpha}\lim_{\delta \to 0} \frac{1}{\delta}  
      \mathbb{P}_{\underline{x}_{\delta,\psi}}
      (X_{3}(\infty)=1) 
      = c_{2}\left(1-\left(\frac{1-c_{2}}{1-c_{1}}\right)^{\frac{2\rho 
            (1-c_{2}) (1-c_{1})}{(c_{2}-c_{1})^2}}\right).
    \end{equation}
  \item If {\color{black}$\psi >\frac{c_{1}}{c_{2}}$}, the fixation
    probability of $AB$ satisfies
    \begin{equation}\label{finalfixresult0}
      \lim_{\alpha \to \infty} 
      \frac{1}{\alpha}\lim_{\delta \to 0} \frac{1}{\delta}  
      \mathbb{P}_{\underline{x}_{\delta,\psi}}
      (X_{3}(\infty)=1) = 0.
    \end{equation}
  \end{enumerate}
\end{theorem}

\begin{theorem}[Fixation time of type $3\equiv AB$]\label{T:fixTime}
  Assume the same situation as in Theorem~\ref{theoremfixprob2}.1.\
  and let $S := \inf\{t: X_3(t)=1\}$ (with $\inf\emptyset =
  \infty$). Then, for all $\varepsilon>0$
  \begin{align*}
    \lim_{\alpha\to\infty} \lim_{\delta\to 0}\mathbb P_{\underline x_{\delta, \psi}}
    \Big(\Big|\frac{\alpha}{\log\alpha} S - \Big(\frac{1-\psi}{c_2-c_1} + \frac{2}{1-c_2}\Big)
    \Big|>\varepsilon \Big|X_3(\infty)=1\Big) = 0.
  \end{align*}
\end{theorem}

\begin{remark}[Additive selection]
  For additive selection, we have $c_{1}+c_{2}=c_3=1$, and
  \eqref{finalfixresult} in this case turns into
  \begin{equation*}
    \lim_{\alpha \to \infty} \frac{1}{{\color{black}2}\alpha}\lim_{\delta \to 0} \frac{1}{\delta}  
    \mathbb{P}_{\underline{x}}(X_{AB}(\infty)=1) = c_{2}\left(1-
      \left(\frac{c_{1}}{c_{2}}\right)^{\frac{2\rho c_{1} c_{2}}{(c_{2}-c_{1})^2}}\right).
  \end{equation*} 
  This limit result matches with the approximation result presented in
  (5) of~\cite{HartfieldOtto2011} for a semi-deterministic model. The
  result on the fixation time from Theorem~\ref{T:fixTime} translates
  for additive selection to
  \begin{align*}
    \lim_{\alpha\to\infty} \lim_{\delta\to 0}\mathbb P_{\underline x_{\delta, \psi}}
    \Big(\Big|\frac{\alpha}{\log\alpha} S - \Big(\frac{2c_2 - (1+\psi)c_1}{c_1(c_2-c_1)}\Big)
    \Big|>\varepsilon \Big| {\color{black}X_3(\infty)=1}\Big) = 0.
  \end{align*}
  Since $2c_2 - (1+\psi)c_1 > 2(c_2-c_1)$, this fixation time is even
  longer than for a single beneficial allele with selection
  coefficient $\alpha_1$.
\end{remark}

\begin{remark}[Interpretation for finite populations]
  The limit results can be used as an approximation for finite
  populations with initial value $x_2 = \delta = 1/N$. Writing $X^N$
  for frequencies in the finite model, $s_i:=\alpha_i/N, i=0,...,3$,
  and inserting this value in \eqref{finalfixresult} leads to the
  approximation
  \begin{equation*}
    \begin{aligned}
      \mathbb{P}_{\underline{x}}(X^N_{3}(\infty)=1) \approx
      2s_{2}\left(1-\left(\frac{1-c_{2}}{1-c_{1}}\right)^{\frac{2\rho
            (1-c_{2}) (1-c_{1})}{(c_{2}-c_{1})^2}}\right).
    \end{aligned}
  \end{equation*}
  In addition, the fixation time from Theorem~\ref{T:fixTime} is
  approximately
  $$ \Big(\frac{1-\psi}{c_2-c_1} + \frac{2}{1-c_2}\Big)\frac{\log N}{s}$$
  generations.
\end{remark}

\begin{remark}[Comparison with results by \cite{YuEth2012} and \cite{Cuthetal2012}]
  In \cite{YuEth2012} and \cite{Cuthetal2012}, a similar model was
  analyzed. More precisely, they study a finite Moran model, but use
  as their main scenario $\psi < c_1/c_2$ and compute an approximation
  of the fixation probability. {\color{black} The case
    $\psi > c_1/c_2$, which we also treat in
    Theorem~\ref{theoremfixprob2} above, in fact is split in two more
    cases. If $\psi>1$, the initial frequency
    $X_{Ab}(0) = c\alpha^{-\psi} \ll \alpha^{-1}$ and therefore,
    $1\equiv Ab$ has a small chance to establish at all. If
    $c_1/c_2 < \psi \leq 1$, both $1 \equiv Ab$ and $2\equiv aB$ have
    a chance to establish but $aB$ cannot reach a macroscopic
    frequency and there is no chance that the recombinant type
    $3\equiv AB$ forms; see below for more heuristics. The case
    $\psi=c_1/c_2$ is more difficult due to the discontinuity, which
    arises between the two parts of Theorem~\ref{theoremfixprob2}.}

  The main difference between {\color{black} the results
    from \cite{Cuthetal2012} } and ours is that they
  formulate the fixation probability in terms of a solution of ODEs
  (see (2.8) in~\cite{Cuthetal2012}), while \eqref{finalfixresult} is
  explicit. This difference comes from the phase where the recombinant
  type $3\equiv AB$ forms. Here, using the full finite Moran model
  seems to be more difficult than working with the Ancestral Selection
  Recombination Graph, which consists of only a subset of all possible
  events arising in the population.
\end{remark}

\subsection{Heuristics}
All results in Theorems~\ref{theoremfixprob2} and~\ref{T:fixTime} can
already be understood heuristically. We will make use of several
intuitions:
\begin{enumerate}
\item If the best type in the population has a small frequency, it can
  be approximated by a supercritical branching process, i.e.\ a
  solution of 
  \begin{align}
    \label{eq:supBr}
    dX = \alpha X dt + \sqrt X dW.
  \end{align}
  Here, $\alpha$ is the fitness advantage against the bulk of the
  population. Recall that the survival probability of the SDE
  \eqref{eq:supBr}, starting in $X_0 = x$ is given by
  $1-e^{-2\alpha x}$. Hence, if $x\sim \alpha^{-\psi}$ for some
  $\psi\in [0,1)$, the fixation probability approaches~1
  {\color{black}, while the fixation probability approaches~0 for $\psi>1$}.\\
  The solution of~\eqref{eq:supBr} can be conditioned to survive. In
  this case, it reaches $\varepsilon>0$ by time approximately
  $\frac{\log\alpha}{\alpha}$.
\item If a type is {\it established} in the population, which means
  that its survival probability is close to~1, but its frequency is
  still small, its frequency can be approximated by the logistic
  equation $dX = \alpha X(1-X) dt$ (where $\alpha$ is the fitness
  advantage against the bulk). In particular, for small $X$, the
  growth of $X$ is exponential with rate $\alpha$. Moreover, the time
  it takes to reach $1-\varepsilon$ starting in $\varepsilon>0$ is
  $O(1/\alpha)$.
\item If a type in the population has small frequency, and its fitness
  disdvantage against the bulk is $\alpha$, it can be approximated by
  a subcritical branching process, i.e.\ a solution of
  \begin{align*}
    dX = -\alpha X dt + \sqrt X dW.
  \end{align*}
  When started in $\varepsilon>0$, it dies out (i.e.\ $X$ hits~0) by
  time approximately $\frac{\log\alpha}{\alpha}$.
\end{enumerate}

\subsubsection*{...on the fixation probability} 
Since the initial configuration is
$X(0) = (1 - \delta - \alpha^{-\psi}, \alpha^{-\psi}, \delta, 0)$,
most individuals are type $0\equiv ab$ at the beginning. The survival
probability of type~$1\equiv Ab$ is approximately
$1-e^{-2\alpha_1 \alpha^{-\psi}} \approx 1$ {\color{black} if $\psi<1$
  and $\approx 0$ if $\psi>1$. In the latter case, fixation of
  $3\equiv AB$ is not possible, so we focus on the former case. We} have approximately
$dX_1 = \alpha_1 X_1 dt$ until $X_1$ reaches some small
$\varepsilon>0$. By the exponential growth,
$X_1 = \alpha^{-\psi} e^{\alpha_1 t}$, so the hitting time of
$\varepsilon$ is not earlier than
\begin{align}\label{eq:t1}
  t_1 := (\psi \log\alpha + \log\varepsilon)/\alpha_1 \approx
  \frac{\psi}{c_1} \frac{\log\alpha}{\alpha}.
\end{align}
Moreover, for small $\delta$, type $2\equiv aB$ has a chance of
$1-e^{-2\alpha_2\delta} \approx 2\alpha_2\delta \approx 2\alpha\delta
c_2$
of surviving (recall that type~2 has to survive only against type~0
since~1 is still in low frequency; note that $\delta$ goes to~0 first
in Theorem~\ref{theoremfixprob2}). If type~2 survives, it follows the
SDE $dX = \alpha_2 X_2 dt + \sqrt{X_2}dW$ (at least until $X_1$
becomes too large), conditioned on survival, so it hits
$\varepsilon>0$ not earlier than at time
\begin{align*}
  t_1' := \frac{\log\alpha_2}{\alpha_2} \approx \frac{1}{c_2}\frac{\log\alpha}{\alpha}.
\end{align*}
Comparing~$t_1$ and $t_1'$, we have to distinguish two cases. We start
with the simpler one, which leads to 2.\ in the Theorem: If
$\psi > \frac{c_1}{c_2}$, we find that $t_1'<t_1$, i.e.\ $X_2$ hits
$\varepsilon$ before $X_1$, and shortly after $t_1'$, we have that
$X_2\approx 1$, so $X_1$ becomes subcritical. In particular, $X_1$
cannot have any macroscopic frequency, and type $3\equiv AB$ has no
chance to form by recombination. This already explains
\eqref{finalfixresult0}. However, if $\psi < \frac{c_1}{c_2}$ (see 1.\
in the Theorem), we have $t_1<t_1'$, i.e.\ the frequency $X_1$ hits
$\varepsilon$ quicker than $X_2$. Shortly after $t_1$, $X_1$ reaches
frequency~1. So, by this time we have
$X_0\approx 0, X_1\approx 1, X_2 \approx e^{(\alpha_2 \psi \log
  \alpha)/\alpha_1}/\alpha_2\approx \alpha^{c_2\psi/c_1 -1}/c_2,
X_3=0$.
In particular, establishment of $X_2$ happened. Then, $\underline X$
can be approximated by $X_0=X_3=0$,
$$ dX_2 = (\alpha_2 - \alpha_1) X_1 X_2 dt, \qquad dX_1 = -dX_2.$$
In words, $X_2$ grows logistically at speed $(\alpha_2 - \alpha_1)$
and $X_1 = 1-X_2$.  The time it takes $X_2$ to reach some
$\varepsilon>0$ is thus
\begin{align}
  \label{eq:t2}
  t_2 = t_1 + \frac{1}{c_2-c_1} \Big(1 - \frac{c_2\psi}{c_1}\Big)\frac{\log\alpha}{\alpha}.
\end{align}
Once $X_2$ reaches some $\varepsilon>0$, a recombination between
$1\equiv Ab$ and $2\equiv aB$ can occur. It is crucial to note that if
a recombinant arises by some time $t$, its chance to survive depends
on $X_1$ and $X_2$. {\color{black} Switching back to a finite
  population of size $N$ for a moment, we denote}  by
$S_t/N$ the chance that a recombinant at time $t$
survives. {\color{black} Then,} the chance that a surviving recombinant
arises is {\color{black} approximately},
{\color{black} 
  \begin{align}\label{eq:rhoI}
    & 1 - \exp\Big( - \frac \rho N  \int N^2 X_1(t) X_2(t) \frac{S_t}{N} dt\Big) = 
      1 - \exp\Big( - \rho I\Big)
      \intertext{with}
    & I = \int X_1(t) X_2(t) S_t dt.
  \end{align}}
Moreover, as can be computed from Theorem~\ref{T:tibd} and
Proposition~\ref{P:49} (recall $\alpha_3\equiv \alpha$),
\begin{align*}
  S_t = 2\frac{(\alpha - \alpha_1)(\alpha - \alpha_2)}{(\alpha - \alpha_1)X_2(t) + (\alpha -
  \alpha_2)X_1(t)}.
\end{align*}
Plugging this into the last display shows that, using
$dX_2 = (\alpha_2 - \alpha_1) X_1 X_2 dt$,
\begin{align*}
  I & = \int X_1(t) X_2(t) S_t dt 
  \\ & = 2\frac{(\alpha - \alpha_1)
       (\alpha - \alpha_2)}{\alpha_2 - \alpha_1} \int \big((\alpha - \alpha_1)x + (\alpha -
       \alpha_2)(1-x)\big)^{-1} dx 
  \\ & = 2\frac{(\alpha - \alpha_1)
       (\alpha - \alpha_2)}{(\alpha_2 - \alpha_1)^2}\log\frac{\alpha - \alpha_1}{\alpha 
       - \alpha_2}.
\end{align*}
Therefore,
\begin{align*}
  & \mathbb P_{\underline x_{\delta,\psi}} (X_3(\infty) = 1) 
    \approx 2\alpha \delta c_2 \big(1 - e^{-\rho I}\big)
    \approx 2\alpha \delta c_2 \Big(1 - 
    \Big(\frac{1 - c_1}{1 - c_2}\Big)^{-2\rho \frac{(1 - c_1)
    (1 - c_2)}{(c_2 - c_1)^2}}\Big),
\end{align*}
and we have shown \eqref{finalfixresult} from
Theorem~\ref{theoremfixprob2} heuristically.

\begin{remark}[Bounds]
  We note that some bounds on the fixation probability can be
  established heuristically as well. Using that
  \begin{align*}
    I & \approx 2\frac{(1-c_2)(1-c_1)\log\big(\frac{1-c_1}{1-c_2}\big)}{(c_2-c_1)^2} \leq 
        2\frac{(1-c_2)(1-c_1)\big(\frac{1-c_1}{1-c_2}-1\big)}{(c_2-c_1)^2} = 2\frac{1-c_1}{c_2-c_1},\\
    I & \approx - 2\frac{(1-c_2)(1-c_1)\log\big(\frac{1-c_2}{1-c_1}\big)}{(c_2-c_1)^2} \geq
        - 2\frac{(1-c_2)(1-c_1)\big(\frac{1-c_2}{1-c_1}-1\big)}{(c_2-c_1)^2} = 2\frac{1-c_2}{c_2-c_1},
  \end{align*}
  we obtain
  \begin{align*}
    2\alpha\delta c_2\Big( 1 - \exp\Big(
    - 2\rho\frac{1-c_2}{c_2 - c_1}\Big)\Big)  \lesssim
    \mathbb P_{\underline x_{\delta,\psi}} (X_3(\infty) = 1) 
    \lesssim
    2\alpha\delta c_2\Big( 1 - \exp\Big(
    - 2\rho\frac{1-c_1}{c_2 - c_1}\Big) \Big).
  \end{align*}
  This also follows from
  $2(\alpha - \alpha_2) \leq S_t\leq 2(\alpha-\alpha_1)$, which
  holds since the recombinant has at least to survive against type~1
  and at most against type~2. Plugging the upper bound into
  \eqref{eq:rhoI}, we obtain again
  \begin{align*}
    \mathbb P_{\underline x_{\delta,\psi}} (X_3(\infty) = 1) 
    & \approx 2\alpha\delta c_2\Big(1 - \exp\Big( 
      - \rho \int X_1(t) X_2(t) S_t dt\Big)\Big) 
    \\ & \leq 2\alpha\delta c_2\Big(1 - \exp\Big( - 2\rho\frac{\alpha - \alpha_1}{\alpha_2 - \alpha_1} 
         \int dX_2\Big) \Big) 
    \\ & 
         \approx 2\alpha\delta c_2\Big(1 - \exp\Big( - 2\rho\frac{1-c_1}{c_2 - c_1}\Big)\Big)
  \end{align*}
  and a similar relation holds for the lower bound.
\end{remark}

\subsubsection*{...on the fixation time} {\color{black} Here, we are in
  the case $\psi < c_1/c_2$.} For developing a heuristics on the
fixation time (see Theorem~\ref{T:fixTime}), we rescale time by the
factor $\alpha/(\log\alpha)$ for the moment, leading to a time-scale
$d\tau = \frac{\alpha}{\log\alpha}dt$. We have already seen above
(see~\eqref{eq:t1}) that $1\equiv Ab$ hits frequency $\varepsilon$ by
time
$$\tau_1 \approx \frac{\psi}{c_1}$$ 
and that type~3 -- it if arises -- arises (see~\eqref{eq:t2}) by time
$$ \tau_2 = \tau_1 + \frac{1}{c_2-c_1} \Big(1 - \frac{c_2\psi}{c_1}\Big).$$
Since the phase where type~$2\equiv aB$ outcompetes type~$1\equiv aB$
only takes $O(1/\alpha)$, we find that at time $\tau_2+$ that
$X_1(\tau_2+)\approx \varepsilon, X_2(\tau_2+)\approx 1-\varepsilon$
for some small $\varepsilon>0$. So, the successful type $3\equiv AB$
has to compete against type $2\equiv aB$, all other types being of
small frequency. The fixation time of type $3\equiv AB$ is thus given
by the classical result for a beneficial allele, which is by time
$$ \tau_3 = \tau_2 + \frac{2}{1-c_2}.$$
In total, we find that{\color{black}, conditional on fixation of
  $3\equiv AB$,}
$$\tau_3 \approx \frac{\psi}{c_1} + \frac{c_1 - c_2\psi}{c_1(c_2-c_1)} + \frac{2}{1-c_2} 
= \frac{1-\psi}{c_2-c_1} + \frac{2}{1-c_2}.$$

\subsubsection*{...on the genetic footprint of the doubly beneficial
  mutant}
Detecting selection from a population sample is a formidable task and
has benefitted from new methods in the last two decades. Such methods
are based on polymorphism data, as reviewed e.g.\ by
\cite{Stephan2016}. Frequently, detecting strong selection is based on
the hitchhiking effect, i.e.\ the reduced neutral genetic diversity
around a beneficial locus at or near the time of its fixation. In the
case of competing sweeps, simulation results from
\cite{KimStephan2003} and \cite{chevinetal2008} report a {\it reduced
  hitchhiking effect} in the case of competing sweeps (relative to the
scenario of a single beneficial allele rising to fixation). This
indicates that the reduction in neutral genetic diversity, caused by
the fast fixation process of the beneficial alleles, is weaker for
competing sweeps. In addition, \cite{chevinetal2008} report an {\it
  increased number of intermediate frequency variants}.

Although we will not contribute to a quantitative understanding of
genetic patterns under competing sweeps in this manuscript, we will
add some ideas how the above findings can be understood. For this,
recall from the heuristics above that type~$3\equiv AB$ only arises
during the time when type~$2\equiv aB$ takes over a population of
mostly type-$1\equiv Ab$-individuals. It is possible that several
type~3 recombinants occuring during this time contribute to fixation
of type~3. Since different recombinants will have different cross-over
points between the $A$- and $B$-locus, the case of multiple
recombinants leads to a haplotype structure between the two selective
loci. Since recombinants arise at nearly the same time, they will rise
in frequency similarly, leading to (i) an increase in intermediate
frequency variants due to the haplotype structure and consequently
(ii) a reduction in the hitchhiking effect.

\section{The Ancestral Selection Recombination Graph}
\label{S:ASRG}
For computing moments of a diffusion such as \eqref{eq:diffsystem}, a
genealogical picture can be extremely helpful (see
e.g. \citealp{Mano2009,PokPfa2012,alkemper2007graphical,EtheridgeEtAl2006}). Since
we study here a scenario with random reproduction, recombination and
selection, all these forces have to be taken into account in the
genealogical picture.

For populations which evolve under selection, a general graph
construction called the ancestral selection graph (ASG), was
introduced by \citeauthor{NeuKrone1997}. It contains all information
about the ancestry of a population sample. The basic idea is best
explained by means of the graphical representation of the Moran
model. When the ancestry of a sample of lines is examined backwards in
time, selective events occur at certain time points. As long as we do
not know the allelic type of the individuals initiating this event, we
cannot decide if the selection event has taken place or not. This
requires information about the allelic types taking part in this
event. To handle this difficulty both potential ancestors are traced
back. In doing so, all possible ancestors of a considered sample are
included in the ASG. Later on when the types in the past are
determined we can decide about the real ancestor of the considered
line.

Related arguments where used for a situation with recombination by
\cite{GrifMarj1997}, which introduced the ancestral recombination
graph. When the ancestry of two loci is considered simultaneously, a
recombination event between the two loci leads to the situation that
the alleles originate from different individuals. Hence to handle the
complete ancestry both ancestors must be traced back.

Both selection and recombination events (viewed backwards in time)
lead to the necessity to split ancestral lines in order to identify
the correct genealogy. In particular, it is possible to combine these
two graphs. This link was done in a very general way by
\cite{DonnellyKurtz1999}, who construct an {\it ancestral influence
  graph} using the lookdown process, a particle representation of
\eqref{eq:diffsystem}; see also \citep{LessardKermany2012}. As a
result, one obtains the ancestral selection recombination graph
(ASRG). Before we start, we provide in the next subsection a
computational tool for the fixation probability which is based on the
Ancestral Selection Recombination Graph.

\subsection{A Markov jump process and the fixation probability}
\noindent
For the proofs of Theorems~\ref{theoremfixprob2} and~\ref{T:fixTime},
we will have to translate the ASRG into a Markov jump process and a
duality relation with the SDE \eqref{eq:diffsystem}.  The process
$\mathcal L$ as defined below arises as a time-reversed version of the
Ancestral Selection Recombination Graph in equilibrium for large
$\alpha$; see the next Subsection for more details. Again, indices
$0,1,2,3$ will correspond to the types in the SDE.

\begin{definition}[The Markov jump process $\mathcal L$\label{def:L0}]
  We define the Markov jump process
  $\mathcal L = (\underline L(t))_{t\geq 0}$ with
  $\underline L(t) = (L_0(t), L_1(t), L_2(t), L_3(t))$: Starting in
  $\underline L(0)$ with
  $L_2(0)=1, L_3(0)=0, L_1(0) \sim \text{Poi}(2 \alpha^{1-\psi}),
  L_0(0) \sim \text{Poi}(2\alpha(1-\alpha^{-\psi}))$,
  the dynamics is as follows: If
  $\underline L(t) = \underline \ell := (\ell_0, \ell_1, \ell_2,
  \ell_3)$,
  jumps occur to (note that
  $\underline e_i = (\delta_{ij})_{j=0,1,2,3}$ and setting $c_0:=0$)
  \begin{align*}
    & \underline \ell + \underline e_i \text{ at rate } \alpha \ell_i, && i=0,...,3,
    \\ 
    & \underline \ell - \underline e_i \text{ at rate } \binom{\ell_i}{2} 
      + \frac 12 \sum_{j\neq i} \ell_i \ell_j(1-c_i+c_j), && i=0,...,3,\\
    %& \underline \ell - \underline e_0 \text{ at rate } \binom{\ell_0}{2} 
    %  + \ell_0 \ell_3 + \frac 12 \ell_0 \ell_2 (1 + c_2) + \frac 12 \ell_0 
    %  \ell_1 (1 + c_1),\\
    %& \underline \ell - \underline e_1 \text{ at rate } \binom{\ell_1}{2} 
    %  + \frac 12 \ell_1 \ell_3 (2 - c_1) + \frac 12 \ell_1 \ell_2 (1 - c_1 + c_2) 
    %  + \frac 12 \ell_1 \ell_0 (1 - c_1),\\
    %& \underline \ell - \underline e_2 \text{ at rate } \binom{\ell_2}{2} 
    %  + \frac 12 \ell_2 \ell_3 (2  -c_2) + \frac 12 \ell_2 \ell_1 (1 - c_2 
    %  + c_1) + \frac 12 \ell_2 \ell_0 (1 - c_2),\\
    %& \underline \ell - \underline e_3 \text{ at rate } \binom{\ell_3}{2} 
    %  + \frac 12 \ell_3 \ell_2 c_2 + \frac 12 \ell_3 \ell_1 c_1,\\
    & \underline \ell + \underline e_i - \underline e_1 - \underline e_2 
      \text{ at rate } \frac 12 \ell_1 \ell_2 \frac{\rho}{\alpha}, && i=0,3,
    \\
    & \underline \ell + \underline e_i - \underline e_0 - \underline e_3 
      \text{ at rate } \frac 12 \ell_0 \ell_3 \frac{\rho}{\alpha}, &&  i=1,2.
  \end{align*}
\end{definition}

\begin{remark}[$\mathcal L$ as a chemical reaction network\label{rem:chem}]
  Actually, the Markov-jump process can be seen as a special form of
  chemical reaction network with mass-action dynamics. Precisely, this
  network is given through the equations
  \begin{align*}
    A_i & \xrightarrow{\alpha} 2A_i, && i=0,1,2,3\\
    A_i + A_j & \xrightarrow{(1 - c_i + c_j )/2} A_j, && i,j=0,1,2,3\\
    A_0 + A_3 & \xrightarrow{\rho/(2\alpha)} A_i, && i=1,2 \\
    A_1 + A_2 & \xrightarrow{\rho/(2\alpha)} A_i, && i=0,3,
  \end{align*}
  where $L_i$ is the number of molecules of species $A_i$,
  $i=0,...,3$. Note that the first reaction looks like binary
  branching dynamics for all species, the second is a special form of
  resampling, and the remaining equations come from recombination
  events within the ASRG.
\end{remark}

\noindent
The next Proposition is fundamental in the proof of
Theorem~\ref{theoremfixprob2}. Its proof is found in
Subsection~\ref{sub:fixProbL}.

\begin{proposition}[Eventual fixation and
  $\mathcal L$\label{P:fixProbL}]
  Let $\mathcal L$ be as in Definition~\ref{def:L0}, $\mathcal X$ as
  in Theorem~\ref{theoremfixprob2} and $S$ as in
  Theorem~\ref{T:fixTime}.
  \begin{enumerate}
  \item Then,
    \begin{equation}\label{pfinalfixresult1}
      \lim_{\alpha \to \infty} 
      \frac{1}{2\alpha}\lim_{\delta \to 0} \frac{1}{\delta}  
      \mathbb{P}_{\underline{x}_{\delta,\psi}}
      (X_{3}(\infty)=1) = \lim_{\alpha\to\infty} \mathbb P(L_j(\infty) = 0, j\neq 3).
    \end{equation}
  \item For $\tau>0$,
    \begin{equation}\label{pfinalfixresult2}
      \lim_{\alpha \to \infty} 
      \frac{1}{2\alpha}\lim_{\delta \to 0} \frac{1}{\delta}  
      \mathbb{P}_{\underline{x}_{\delta,\psi}}
      \Big(\frac{\alpha}{\log\alpha}S < \tau\Big) = \lim_{\alpha\to\infty} 
      \mathbb P\Big(L_j\Big(\tau \frac{\log\alpha}{\alpha}\Big)=0, j\neq 3\Big).
    \end{equation}
  \end{enumerate}
\end{proposition}

\subsection{Construction of the ASRG and duality}
We will define the ASRG first, and then obtain some basic results and
the precise connection to \eqref{eq:diffsystem}. Then, we give the
connection to the jump process from Definition~\ref{def:L0}.

In order to distinguish between the {\it forwards time} $t$ and the
genealogical time, the {\it backwards (genealogical) time} $\beta$ is
introduced; see Figure~\ref{fig:split} for an illustration.

\begin{figure}
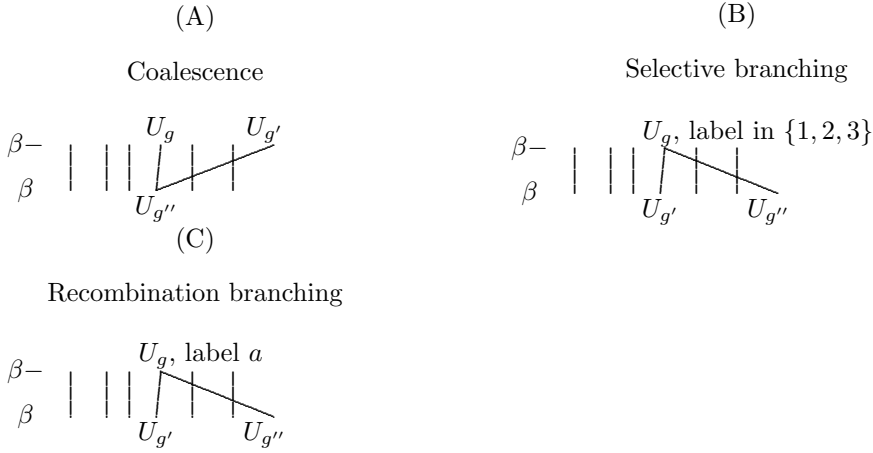

  \parbox{5cm}{\centering{(A)}\\[2ex]Coalescence\\[3ex]
    \beginpicture
    \setcoordinatesystem units <.6cm,.6cm>
    \setplotarea x from 0 to 6, y from -.5 to 1.5
    \put {$\beta-$} [cC] at 0 1
    \put {$\beta$} [cC] at 0 0
    \put {$U_{g}$} [cC] at 3 1.35
    \put {$U_{g'}$} [cC] at 5.3 1.35
    \put {$U_{g''}$} [cC] at 2.9 -.4
    \plot 1 0 1 1 /
    \plot 1.8 0 1.8 1 /
    \plot 2.3 0 2.3 1 /
    \plot 3.7 0 3.7 1 /
    \plot 4.6 0 4.6 1 /
    \plot 2.9 0 5.5 1 /
    \plot 2.9 0 3 1 /
    \endpicture}
  \parbox{6cm}{\centering{(B)}\\[2ex]Selective branching\\[3ex]
    \beginpicture
    \setcoordinatesystem units <.6cm,.6cm>
    \setplotarea x from 0 to 7, y from -.5 to 1.5
    \put {$\beta-$} [cC] at 0 1
    \put {$\beta$} [cC] at 0 0
    \put {$U_{g}$, label in $\{1,2,3\}$} [lC] at 2.5 1.35
    \put {$U_{g''}$} [cC] at 5.3 -.4
    \put {$U_{g'}$} [cC] at 2.9 -.4
    \plot 1 0 1 1 /
    \plot 1.8 0 1.8 1 /
    \plot 2.3 0 2.3 1 /
    \plot 3.7 0 3.7 1 /
    \plot 4.6 0 4.6 1 /
    \plot 3 1 5.5 0 /
    \plot 2.9 0 3 1 /
    \endpicture}
  \parbox{5cm}{\centering{(C)}\\[2ex]Recombination branching\\[3ex]
    \beginpicture
    \setcoordinatesystem units <.6cm,.6cm>
    \setplotarea x from 0 to 6, y from -.5 to 1.5
    \put {$\beta-$} [cC] at 0 1
    \put {$\beta$} [cC] at 0 0
    \put {$U_{g}$, label $a$} [lC] at 2.5 1.35
    \put {$U_{g''}$} [cC] at 5.3 -.4
    \put {$U_{g'}$} [cC] at 2.9 -.4
    \plot 1 0 1 1 /
    \plot 1.8 0 1.8 1 /
    \plot 2.3 0 2.3 1 /
    \plot 3.7 0 3.7 1 /
    \plot 4.6 0 4.6 1 /
    \plot 3 1 5.5 0 /
    \plot 2.9 0 3 1 /
    \endpicture}
  \caption{\label{fig:split} If a coalescing event (A), a selective
    branching event (B) or a recombination branching event (C) occurs
    by time $\beta$, we connect the lines within the ASRG according to
    the rules as given in Definition~\ref{def:ASRG2}. In all cases,
    $U_{g}$ is uniformly distributed on $[0,1]$, and updated upon any
    event for the affected lines. For branching event, labels in
    $\{1,2,3,a\}$ indicate which kind of event happens.}
\end{figure}

\begin{definition}[ASRG] \label{def:ASRG2}
  \begin{enumerate}
  \item For $k\in\mathbb N$, we define a particle model
    $\mathcal A = \mathcal A^k = (\mathcal A^k_\beta)_{\beta\geq 0}$,
    which takes values in
    $$ E := \bigcup_{k=1}^\infty E_k, \qquad E_k:=\{A \subset [0,1]: \#A = k\},$$
    the set of finite subsets of $[0,1]$ and $A^k_0 \in E_k$. Elements
    of $\mathcal A^k_\beta$ are called particles. We set
    $\mathcal A_0^k = \{U_1,...,U_k\} \in E_k$ for some (on $[0,1]$)
    uniformly distributed and independent random variables
    $U_i, i=1,...,k$. The dynamics of $\mathcal A^k$ is as follows:
    \begin{enumerate}
    \item Coalescence: Every (unordered) pair of particles is replaced
      at rate~1 by a particle with a label that is uniformly
      distributed on $[0,1]$ and independent of everything else.
    \item Branching: Every particle is replaced at rate $\alpha+\rho$
      by two particles, with labels that are uniformly distributed on
      $[0,1]$ and independent of everything else. 
      %Call the first
      %particle {\it incoming} and the second {\it continuing}.
    \end{enumerate}
    Along the path of $\mathcal A^k$, we mark each branching event
    with a label in $\{1,2,3,a\}$, namely
    \begin{equation}\label{labelprob}
      \begin{aligned}
        1 & \quad\text{with probability} \quad \frac{\alpha_{1}}{\alpha+\rho} \\
        2 & \quad\text{with probability}  \quad \frac{\alpha_{2}-\alpha_{1}}{\alpha+\rho} \\
        3 & \quad\text{with probability}  \quad \frac{\alpha-\alpha_{2}}{\alpha+\rho} \\
        a & \quad \text{with probability} \quad
        \frac{\rho}{\alpha+\rho}.
      \end{aligned}
    \end{equation}
    The branching events with a label in $\{1,2,3\}$ are called {\it
      selective (branching) events} and the events with the letter $a$
    are called {\it recombination (branching) events}. For a selective
    branching event, the branching particle will be denoted the {\it
      outgoing} particle and we mark one new particle as the {\it
      incoming}, the other one as the {\it continuing particle}. For a
    recombination event, mark one new particle as the {\it
      $A/a$-particle}, the other one as the {\it $B/b$-particle}.  We
    denote the particle system which includes all these marks by
    $\mathcal B^k$, which we refer to by the {\it Ancestral Selection
      Recombination Graph} or {\it ASRG} for short.
  \item For the particle system
    $\mathcal B^k = (\mathcal B_\beta^k)_{\beta\geq 0}$, consider the
    line-counting process $(K^k_\beta)_{\beta\geq 0}$, which starts in
    $K^k_0=k$ and jumps from $k$ to
    \begin{equation}\label{linecountrates}
      \begin{aligned}
        k-1 & \quad \text{at rate} \quad  q_{k,k-1}:=\binom{k}{2} \\
        k+1 & \quad \text{at rate} \quad q_{k,k+1}:=(\alpha+\rho)k.
      \end{aligned}
    \end{equation}
    {\color{black}Note that we frequently refer to trajectories of
      particles as ancestral lines, therefore we call
      $(K^k_\beta)_{\beta\geq 0}$ a {\it line}-counting process rather
      than a {\it particle}-counting process.}
  \item Let $\underline x := (x_0,x_1,x_2,x_3) \in \Sigma_3$, the
    three-dimensional simplex and $\tau>0$. Based on the ASRG
    $\mathcal B^k$, we define random variables
    $J_1(\tau),...,J_k(\tau) \in \{0,1,2,3\}$. (Recall that time for
    the ASRG is running {\it backward}.) First, color all particles in
    $\mathcal B^{k}_\tau$ independently with color $i$ with
    probability $x_i$, $i=0,1,2,3$. From here on, work {\it forwards}
    through the graph, such that types are inherited along coalescence
    events; see Figure \ref{fig:split1}(A).  Upon a selective
    branching event, do the following: If the label of the selective
    branching event is $i$, the incoming line is of type $j$, and the
    continuing line of of type $k$, the outgoing line is of type (see
    Figure~\ref{fig:split1}(B))
    \begin{align}\label{eq:splitSel}
      1_{j\geq i} j + 1_{j<i} k.
    \end{align}
    (This means that a type with a number higher than the mark is
    inherited along the incoming line, but if this is not the case,
    the type of the continuing line is inherited.) Upon a
    recombination branching event, do the following: the $A/a$-locus
    is inherited along the $A/a$-particle, the $B/b$-locus is
    inherited along the $B/b$-particle. Precisely, use Table
    \ref{tab2} for the types; see also Figure~\ref{fig:split1}(C). By
    this procedure, we obtain
    $J_1(\tau),...,J_k(\tau) \in \{0,1,2,3\}$, which are the types of
    the $k$ initial particles of $\mathcal B^k$ at time $\beta=0$. We
    denote the distribution of $(J_i(\tau))_{i,\tau}$ by
    $\mathbb P_{\underline x}$.
  \end{enumerate}
\end{definition}

\begin{table}
  \centering
  \begin{tabular}{lll}
    $A/a$-particle & $B/b$-particle & outgoing particle \\
    0 or 2 & 2 or 3 & 2 \\
    0 or 2 & 0 or 1 & 0 \\
    1 or 3 & 2 or 3 & 3 \\
    1 or 3 & 0 or 1 & 1 \\
  \end{tabular}
  \caption{\label{tab2} Lookup Table for recombination branching events. The $A/a$-locus is inherited along 
    the $A/a$-particle, the $B/b$-locus along the $B/b$-particle.}
\end{table}

\begin{figure}
  \parbox{5cm}{\centering{(A)}\\[2ex]Coalescence\\[3ex]
    \beginpicture
    \setcoordinatesystem units <.6cm,.6cm>
    \setplotarea x from 0 to 6, y from 1.5 to 1.5
    \put {$i$} [cC] at 3 1.35
    \put {$i$} [cC] at 5.3 1.35
    \put {$i$} [cC] at 2.9 -.9
    \plot 2.9 0 5.5 1 /
    \plot 2.9 0 3 1 /
    \plot 2.9 0 2.9 -0.4 /
    \endpicture}
  \parbox{6cm}{\centering{(B)}\\[2ex]Selective branching\\[3ex]
    \beginpicture
    \setcoordinatesystem units <.6cm,.6cm>
    \setplotarea x from 0 to 5, y from -.5 to 1.5
    \put {$1_{j\geq i}j + 1_{j<i}k$} [cC] at 3 2
    \put {$i$} [cC] at 3.3 1.3
    \put {$j$} [cC] at 5.3 -.4
    \put {$k$} [cC] at 2.9 -.4
    \put {incoming} [cC] at 5.6 .7
    \put {continuing} [cC] at 1.3 .4
    \plot 3 1 5.5 0 /
    \plot 3 1 2.9 0 /
    \plot 3 1 3 1.7 /
    \endpicture}
  \parbox{5cm}{\centering{(C)}\\[2ex]Recombination branching\\[3ex]
    \beginpicture
    \setcoordinatesystem units <.6cm,.6cm>
    \setplotarea x from 0 to 7, y from -.5 to 1.5
    \put {see Table~\ref{tab2}} [cC] at 3 2.1
    \put {$a$} [cC] at 3.3 1.3
    \put {$j$} [cC] at 5.3 -.4
    \put {$k$} [cC] at 2.9 -.4
    \put {$B/b$-line} [cC] at 5.5 .7
    \put {$A/a$-line} [cC] at 1.3 .4
    \plot 3 1 5.5 0 /
    \plot 3 1 2.9 0 /
    \plot 3 1 3 1.7 /
    \endpicture}
  \caption{\label{fig:split1} If a coalescing event (A), a selective
    branching event (B) or a recombination branching event (C) occurs
    by time $\beta$, we connect the lines within the ASRG according to
    the rules as given in Definition~\ref{def:ASRG2}. In all cases,
    $U_{g}$ is uniformly distributed on $[0,1]$, and updated upon any
    event for the affected lines. For branching event, labels in
    $\{1,2,3,a\}$ indicate which kind of event happens.}
\end{figure}

\begin{proposition}[Duality\label{P:duality}]
  Let $\mathcal X$ be as in \eqref{eq:diffsystem}, and
  $J_1(\tau),...,J_k(\tau)$ as in Definition~\ref{def:ASRG2}.3 for
  $\tau>0$. Then, the duality relation
  \begin{align}\label{eq:duality}
    \mathbb E_{\underline x}\Big[\prod_{i=1}^k X_{j_i}(\tau)\Big] = \mathbb P_{\underline x}[J_1(\tau) 
    = j_1,...,J_k(\tau) = j_k], \qquad 
    j_1,...,j_k \in \{0,1,2,3\}
  \end{align}
  holds.
\end{proposition}

\begin{proof}
  As shown in \cite[{\color{black}Section~8}]{DonnellyKurtz1999}, an
  ASRG (as a special case of their ancestral influence graph, AIG) is
  the graph of possible ancestors of an infinitely large population,
  whose evolution of frequencies is given by \eqref{eq:diffsystem}.
  {\color{black} See also \cite{Fearnhead2002} for the definition of
    the AIG including selection acting at two linked loci, which is
    all what we need here.}  Note that the left-hand side of
  \eqref{eq:duality} is the probability that in a sample of size $k$,
  taken at time $\tau$ from the population which evolves according
  to~\eqref{eq:diffsystem}, the $i$th pick is of type $j_i$,
  $i=1,...,k$. The right-hand side follows all possible ancestries of
  this sample back until time~0, and then decides forwards in time
  which types are inherited (i.e.\ which selective events take place
  and where the $A/a$- and $B/b$-locus was inherited in a
  recombination event. The right hand side therefore gives the
  probability that the ASRG and the labeling is such that types
  $j_1,...,j_k$ appear in the sample. Since the types of the lines at
  time $\beta = \tau$ in the ASRG are used according to the starting
  distribution $\underline{x}$ (see Definition \ref{def:ASRG2}), the
  result follows.
\end{proof}

\subsection{Fixation and the ASRG}
For computing the fixation probability of type $3\equiv AB$, we will
require the equilibrium, denoted $\Pi$, for the
line-counting process $(K_\beta^k)_{\beta \geq 0}$.

\begin{proposition}[The probability of fixation and the ASRG\label{P:fix}]
  Let $\Pi$ be Poisson-distributed with parameter
  {\color{black}$2(\alpha + \rho)$}, conditioned to be positive, i.e.
  \begin{align}\label{eq:Pi}
    \mathbb P(\Pi = k) = \frac{e^{-2(\alpha+\rho)}}{1-e^{-2(\alpha + \rho)}}\frac{2^k(\alpha + \rho)^k}{k!}.
  \end{align}
  \begin{enumerate}
  \item For all $k$, we have
    $K_\beta^k \xRightarrow{\beta\to\infty}\Pi$, i.e.\ $\Pi$ is the
    unique equilibrium for the line-counting process. This equilibrium
    is reversible.
  \item \sloppy Denote by $(K_\beta^\Pi)_{0\leq \beta\leq \tau}$ the
    line-counting process, started in $\Pi$, and let
    $J_1(\tau),...,J_\Pi(\tau)$ be as in
    Definition~\ref{def:ASRG2}.3. Then,
    \begin{align}\label{eq:fixJ}
      \mathbb P_{\underline x}(X_i(\infty)=1) = \lim_{\tau\to\infty} \mathbb P_{\underline x}(J_1(\tau) = \cdots = 
      J_\Pi(\tau) = i), \qquad i=0,1,2,3.
    \end{align}
  \end{enumerate}
\end{proposition}

\begin{proof}
  1. The line-counting process $(K_\beta^k)_{\beta\geq 0}$ is a
  birth-death process with birth rate $(\alpha + \rho)k$ and death
  rate $\binom k 2$. The unique, reversible equilibrium of this
  process is given by \eqref{eq:Pi}; see e.g.\
  \cite{greven2014fixation}, Proposition~3.3. 

  2. Consider some $\tau,\tau'\to\infty$ with
  $\tau' - \tau \to\infty$. Then, if $\mathcal A^k$ starts at time
  $\beta = 0$ with $k$ lines, it has (in the limit of large times)
  $\Pi$ lines at time $\beta = \tau' - \tau$.  From here on, we use a
  restart argument and start the ASRG with $\Pi$ lines at time
  $\tau' - \tau$, which we take as the initial time $\beta = 0$ of the
  restarted ASRG. Since $\tau$ is large, in the ASRG
  $(\mathcal A^\Pi_\beta)_{0\leq \beta \leq \tau}$ we find with high
  probability a time $0\leq \beta'\leq \tau$ with
  $\#\mathcal A^k = 1$. The type of this single particle is inherited
  to the whole ASRG when viewed forwards in time. Therefore, for large
  $\tau$, the ASRG is monotype, and the type is inherited along the
  graph forwards in time, so fixation of type $i$ occurs iff all $\Pi$
  lines of the ASRG at time $0$ carry type $i$. Together with
  Proposition~\ref{P:duality}, we have proven the result.
\end{proof}

Since the equilibrium $\Pi$ is reversible, there is actually another
way to compute the fixation probability. For this, we require another
Markov jump process.

\begin{definition}[The Markov jump process
  $\widetilde{\mathcal L}$\label{def:L}]
  Let $\underline x\in \Sigma_3$ be given. We define the Markov jump
  process
  $\widetilde{\mathcal L} = (\widetilde{\underline L}(t))_{t\geq 0}$
  with
  $\widetilde{\underline L}(t) = (\widetilde L_0(t), \widetilde
  L_1(t), \widetilde L_2(t), \widetilde L_3(t))$.
  For the initial state, take
  $\widetilde{\underline L}_i(0) \sim \text{Poi}(2(\alpha+\rho)x_i)$,
  $i=0,1,2,3$, independently, and condition on
  $\widetilde L_0(0) + \widetilde L_1(0) + \widetilde L_2(0) +
  \widetilde L_3(0) > 0$.
  The dynamics is as follows: If
  $\widetilde{\underline L}(t) = \underline \ell := (\ell_0, \ell_1,
  \ell_2, \ell_3)$,
  jumps occur to (note that
  $\underline e_i = (\delta_{ij})_{j=0,1,2,3}$)
  \begin{align*}
    & \underline \ell + \underline e_i \text{ at rate } (\alpha + \rho) \ell_i, i=0,1,2,3,
    \\ 
    & \underline \ell - \underline e_i \text{ at rate } \binom{\ell_i}{2} + 
      \frac 12 \sum_{j\neq i} \ell_i \ell_j \frac{\alpha - \alpha_i + \alpha_j}{\alpha+\rho} 
      + \frac 12 \ell_i (\ell_1 + \ell_2) \frac{\rho}{\alpha + \rho}, \quad i=0,3\\
    & \underline \ell - \underline e_i \text{ at rate } \binom{\ell_i}{2} + 
      \frac 12 \sum_{j\neq i} \ell_i \ell_j \frac{\alpha - \alpha_i + \alpha_j}{\alpha+\rho} 
      + \frac 12 \ell_i (\ell_0 + \ell_3) \frac{\rho}{\alpha + \rho}, \quad i=1,2\\
    & \underline \ell + \underline e_i - \underline e_1 - \underline e_2 
      \text{ at rate } \frac 12 \ell_1 \ell_2 \frac{\rho}{\alpha + \rho} \text{ for }i=0,3,
    \\
    & \underline \ell + \underline e_i - \underline e_0 - \underline e_3 
      \text{ at rate } \frac 12 \ell_0 \ell_3 \frac{\rho}{\alpha + \rho} \text{ for }i=1,2.
  \end{align*}
  We denote the distribution of $\widetilde{\mathcal L}$ by
  $\mathbb P_{\underline x}$.
\end{definition}

\begin{proposition}[The probability of fixation and the reversed ASRG\label{P:fixL3}]
  Let $\Pi$ be as in Proposition~\ref{P:fix},
  $J_1(\tau),...,J_\Pi(\tau)$ as in Definition~\ref{def:ASRG2} and
  $\widetilde{\mathcal L}$ as in Definition~\ref{def:L}. Then,
  \begin{align*}
    \mathbb P_{\underline x}(J_1(\tau) = \cdots = J_\Pi(\tau) = i) = 
    \mathbb P_{\underline x}(\widetilde L_j(\tau) = 0, j\neq i).
  \end{align*}
  In particular, combining with \eqref{eq:fixJ}
  \begin{align}\label{eq:fixL3}
    \mathbb P_{\underline x}(X_i(\infty)=1) = \mathbb P_{\underline x}(\widetilde L_j(\infty) = 0, j\neq i).
  \end{align}
\end{proposition}

\begin{proof}
  We only have to show the first identity. Since the equilibrium $\Pi$
  of the line-counting process $(K_{\beta})_{0\leq \beta\leq \tau}$ is
  reversible, it is in distribution equivalent to construct the ASRG
  forwards in time and using the same rates for coalescence and
  branching. Hence, the (unlabelled) ASRG $\mathcal A^\Pi$ can as well
  be constructed forwards in time. Denote the time-reversed ASRG by
  $\widetilde{\mathcal A}^\Pi$. In this time direction, colorings of
  particles can be performed at time $\beta = \tau$, $t=0$ already
  when initializing the ASRG. The difference is that coalescence
  events for $\widetilde{\mathcal A}^\Pi$ are branching events for
  $\mathcal A^\Pi$ and vice versa. In $\widetilde{\mathcal A}^\Pi$,
  coalescence events obtain the labels in $\{1,2,3,a\}$ with
  probabilities as in \eqref{labelprob}, and upon a coalescence event
  with label $\{1,2,3\}$, call one line (picked at random from both
  lines) incoming and the other one continuing; upon a coalescence
  event with label $a$, call one line the $A/a$-line, the other one
  the $B/b$-line. Then, (recall that types are known in
  $\widetilde{\mathcal A}^\Pi$ since we have constructed the graph
  forwards in time) types in coalescence events are decided using
  either \eqref{eq:splitSel} or Table~\ref{tab2}. The resulting
  transitions are given in the definition of the process
  $\widetilde{\mathcal L}$. {\color{black} We will not give all details
    that this is the case, but provide a detailed explanation of one
    term as well as a calculation that the line-counting process
    derived from $\widetilde{\mathcal L}$ has the correct
    dynamics. First, } as an example, consider the transition
  $\underline \ell \to \underline \ell - \underline e_1$.  This occurs
  if either two lines of type~1 coalesce (no matter which label the
  coalescence event has), or if a line of type~1 coalesces with a
  second line. If the second line is of type~3 (which occurs at total
  rate $\ell_1\ell_3$), the chance that the coalescence event has mark
  in $\{1,2,3\}$ is $\alpha/(\alpha + \rho)$. If the line of type~1 is
  the incoming line, the outgoing line is of type~3 if the mark is in
  $\{2,3\}$, which has probability
  $(\alpha - \alpha_1)/(\alpha + \rho)$. If the line of type~1 is the
  continuing line and the line of type~3 is the incoming line, then
  the outgoing line is of type~3 in all cases. In total, the rate of a
  decrease of lines of type~1 by a coalescence with a line of type~3
  is
  $$\frac 12 \ell_1\ell_3 \Big(\frac{\alpha}{\alpha + \rho} + \frac{\alpha - \alpha_1}{\alpha + \rho}\Big);$$
  see Definition~\ref{def:L}. Similar arguments lead to all other
  terms in the definition of $\widetilde{\mathcal L}$.  {\color{black}
    Second, we check if the line-counting process derived from
    $\widetilde{\mathcal L}$ is indeed given by
    \eqref{linecountrates}, i.e.\ the sum total of transition rates
    are correct. Here, the total number of particles jumps from
    $\ell = \ell_0 + \ell_1 + \ell_2 + \ell_3$ to $\ell+1$ at rate
    $(\alpha + \rho)\ell$, and to $\ell-1$ at rate
    \begin{align*}
      \sum_{i=0}^3 & \binom{\ell_i}{2} + \frac 12 \sum_{i,j=0 \atop i\neq j}^3\ell_i \ell_j 
                     \frac{\alpha - \alpha_i + \alpha_j}{\alpha + \rho} + \big((\ell_0 + \ell_3)(\ell_1 + \ell_2) + (\ell_1 \ell_2 + \ell_0 \ell_3)\big)\frac{\rho}{\alpha + \rho}
      \\ & = \sum_{i=0}^3 \binom{\ell_i}{2} + \frac 12 \sum_{i,j=0 \atop i\neq j}^3\ell_i \ell_j 
           \frac{\alpha}{\alpha + \rho} + \frac 12 \sum_{i,j=0 \atop i\neq j}^3\ell_i \ell_j 
           \frac{\rho}{\alpha + \rho} 
      \\ & = \frac{(\ell_0 + \cdots + \ell_3)(\ell_0 + \cdots + \ell_3-1)}{2} = \binom \ell 2.
    \end{align*}
    Altogether, it follows that } the types in
  $\widetilde{\mathcal L}$ at time $t=\tau$ are in distribution the
  same as the types in $J_1,...,J_\Pi$. This implies the assertion.
\end{proof}

\subsection{Proof of Proposition~\ref{P:fixProbL}}
\label{sub:fixProbL}
From~\eqref{eq:fixL3} we know how to compute fixation probabilities
from $\widetilde{\mathcal L}$. However, the process $\mathcal L$ used
in Proposition~\ref{P:fixProbL} differs from $\widetilde{\mathcal L}$
in two respects. First, its initial distribution fixes $L_2(0)=1$
rather than using some $\underline x\in\Sigma_3$. Second, the dynamics
of $\mathcal L$ and $\widetilde{\mathcal L}$ are only the same for
$\alpha\to\infty$. In order to compare the initial conditions, we need
the following simple lemma.

\begin{lemma}[Total variation distance of $\Pi, \Psi_n$\label{l:tv}]
  Let $\Pi = \Pi_{\alpha,\rho}$ {\color{black}with distribution
    $\mathcal L(\Pi)$} be as in Proposition~\ref{P:fix},
  $\Psi_n = {\color{black}\delta_n \ast} \text{Poi}(2\alpha)$ for some
  $n\in \mathbb Z$.  Then, for large $\alpha$, the total variation
  distance $d_{TV}$ obeys
  $$d_{TV}(\mathcal L(\Pi), \Psi_n) = o(1).$$
\end{lemma}

\begin{proof}
  {\color{black} Let $X\sim \Psi_0$. Considering $n=0$ first, recall
    that $\Pi$ is the Poisson-distribution with parameter
    $2(\alpha + \rho)$, conditioned to be positive. Therefore, since
    $\mathbf E[s^X] = e^{-2\alpha(1-s)}$ for $s\in(0,1)$,
    \begin{align*}
      d_{TV}(\mathcal L(\Pi), & \Psi_0) 
                                = o(1) + e^{-2\alpha} \sum_{k=0}^\infty \Big| \frac{(2\alpha)^k}{k!} 
                                - e^{-2\rho} \frac{(2(\alpha + \rho))^k}{k!}\Big|
      \\ & =  o(1) + e^{-2\alpha} \sum_{k=0}^\infty \frac{(2\alpha)^k}{k!} \Big| 1 - e^{-2\rho} 
           \Big(1 + \frac\rho\alpha\Big)^k \Big| 
      \\ & = o(1) + \mathbf E\Big[ \Big| 1 - e^{-2\rho} 
           \Big(1 + \frac\rho\alpha\Big)^X \Big|\Big]
           \leq o(1) + \mathbf E\Big[ \Big( 1 - e^{-2\rho} 
           \Big(1 + \frac\rho\alpha\Big)^X \Big)^2\Big]^{1/2}
      \\ & = o(1) + \Big( 1 - 2e^{-2\rho} e^{-2\alpha(1 - (1+\rho/\alpha))} + e^{-4\rho}e^{-2\alpha (1 - (1+\rho/\alpha)^2)}\Big)^{1/2}
      \\ & = o(1) + \Big(e^{2\rho^2/\alpha} - 1\Big)^{1/2} = o(1).
    \end{align*}
    Now, }using a recursion and the triangle inequality, it suffices
  to prove that $d_{TV}(\Psi_{i}, \Psi_{i+1}) = o(1)$,
  $i\in\mathbb Z$. We compute,
  \begin{align*}
    d_{TV}(\Psi_{i}, \Psi_{i+1}) & = e^{-2\alpha}\Big( 1 + \sum_{k=1}^\infty \Big| \frac{(2\alpha)^{k}}{k!} 
                                   - \frac{(2\alpha)^{k-1}}{(k-1)!}\Big|\Big)
    \\ & =  e^{-2\alpha} \sum_{k=0}^\infty \frac{(2\alpha)^{k}}{k!} \Big| 1 
         - \frac{k}{2\alpha}\Big| + o(1)
    \\ & =  \mathbb E\Big[\Big|1 - \frac{{\color{black}X}}{2\alpha}\Big|\Big] + o(1)
         \leq \mathbb E\Big[\Big(1 - \frac{{\color{black}X}}{2\alpha}\Big)^2\Big]^{1/2} + o(1)
         = o(1).
  \end{align*}
\end{proof}

\begin{proof}[Proof of Proposition~\ref{P:fixProbL}]
  1. Combining with~\eqref{eq:fixL3}, we have to show that for
  $\underline x = (1 - \alpha^{-\psi} - \delta, \alpha^{-\psi},
  \delta, 0)$
  and $L$ starting as in Definition~\ref{def:L0},
  \begin{align}\label{eq:816}
    \lim_{\alpha\to\infty}\lim_{\delta\to 0}\frac{1}{2\alpha\delta} 
    \mathbb P_{\underline x_{\delta,\psi}}(\widetilde L_j(\infty) = 0, j\neq i) 
    =   \lim_{\alpha\to\infty}\mathbb P(L_j(\infty) = 0, j\neq i).
  \end{align}
  We have to show that the total variation distance between
  $\widetilde{\mathcal L}$ and $\mathcal L$ is small for large
  $\alpha$ (and in the limit $\delta\to 0$). Therefore, we have to
  compare both, the initial condition and the dynamics of the two
  processes.

  \noindent {\it Initial condition:} In $\widetilde{\mathcal L}$, we
  mark the $\Pi$ lines in {\color{black}$\widetilde L(0)$}
  independently with probabilities
  $1-\alpha^{-\psi}-\delta, \alpha^{-\psi}, \delta$ and $0$ with types
  $0,1,2,3$, respectively.  In ${\mathcal L}$, we mark the $\Psi_1$
  lines such that $L_2(0)=1$ and all other lines are marked with
  $0, 1, 3$ independently with probabilities
  $1-\alpha^{-\psi}, \alpha^{-\psi}$ and 0, respectively.

  As we see from Lemma~\ref{l:tv}, the total variation distance for
  the total number of lines if we use $\Psi_0$ instead of $\Pi$ lines
  at time~0 is negligible for large $\alpha$. In addition, in the
  limit $\delta\to 0$, a necessary condition for fixation
  {\color{black} of type $3\equiv AB$} to occur is that at least one
  particle is marked by $2\equiv aB$. Therefore, in $\widetilde L(0)$,
  this has approximate probability
  \begin{align*}
    \lim_{\alpha\to\infty} 
    & \lim_{\delta\to 0}\frac{1}{2\alpha\delta} \mathbb P_{\underline x_{\delta,\psi}}(\widetilde L_j(\infty) = 0, j\neq 3) 
    \\ & = \lim_{\alpha\to\infty}\lim_{\delta\to 0}\frac{1}{2\alpha\delta} \mathbb P_{\underline x_{\delta,\psi}}
         (\widetilde L_j(\infty) = 0, 
         j\neq 3 | \widetilde L_2(0)>0) \cdot \mathbb P_{\underline x_{\delta,\psi}}(\widetilde L_2(0)>0)
    \\ & = \lim_{\alpha\to\infty}\lim_{\delta\to 0}\frac{1}{2\alpha\delta} \mathbb P_{\underline x_{\delta,\psi}}
         (\widetilde L_j(\infty) = 0, 
         j\neq 3 | \widetilde L_2(0)=1) \cdot  \frac{1 - e^{-2\alpha\delta}}{1-e^{-2(\alpha + \rho)}}
    \\ & =  \lim_{\alpha\to\infty}\mathbb P(\widetilde L_j(\infty) = 0, 
         j\neq 3 | \widetilde L(0) = L(0)),
  \end{align*}
  where we have used Lemma~\ref{l:tv} in the last step.
  
  \noindent {\it Dynamics:} Note that the dynamics of $\mathcal L$
  arises from the dynamics of $\widetilde{\mathcal L}$, if we only
  take the leading terms in the rates of all transitions into
  account. For example,
  $\underline \ell \to \underline \ell-\underline e_1$ takes place due
  to encounters of types~0 and~1 at a rate
  $$\frac 12 \ell_1\ell_0\frac{\alpha - \alpha_1 + \rho}{\alpha +
    \rho} = \frac 12 \ell_1\ell_0 (1 - c_1)(1 + o(1))$$
  as $\alpha\to\infty$. Since the event of fixation has a
  {\color{black} continuous} dependence on the
  parameters $c_1, c_2$, we can bound fixation probabilities and times
  if we slightly change these parameters around the
  limits.\\
  2. Here, note that \eqref{eq:fixL3} does not hold for finite times,
  so we have to adapt our reasoning. First, note that \eqref{eq:816}
  still holds at finite times (since the initial conditions as well as
  the dynamics can be coupled as above).  Now, we adopt the argument
  of the proof of Lemma~2.5 in \cite{PokPfa2012}. (However, note that
  in the present paper, we are not working with conditional
  probabilities, such that the correction as in Remark 3.17 of
  \cite{greven2014fixation} does not apply here.)

  By Propositions~\ref{P:duality} and~\ref{P:fixL3}, we have that
  \begin{align*}
    \mathbb P_{\underline x}\Big( \frac{\alpha}{\log\alpha}S<\tau\Big) 
    = \mathbb E_{\underline x}\Big[X_{3}\Big(\tau \frac{\log\alpha}{\alpha}\Big)^\infty\Big] 
    = \mathbb P_{\underline x}\Big(J_1\Big(\tau \frac{\log\alpha}{\alpha}\Big) = J_2\Big(\tau \frac{\log\alpha}{\alpha}\Big) 
    = \cdots = 3\Big).
  \end{align*}
  Clearly, by monotonicity in the number of particles,
  \begin{align*}
    \mathbb P_{\underline x}\Big(J_1\Big(\tau \frac{\log\alpha}{\alpha}\Big) = J_2\Big(\tau \frac{\log\alpha}{\alpha}\Big) 
    = \cdots = 3) 
    & \leq 
      \mathbb P_{\underline x}
      \Big(J_1\Big(\tau \frac{\log\alpha}{\alpha}\Big) = \cdots = J_\Pi\Big(\tau 
      \frac{\log\alpha}{\alpha}\Big) = 3\Big) \\ & = \mathbb P_{\underline x}\Big(\widetilde 
                                                   L_j\Big(\tau\frac{\log\alpha}{\alpha}\Big) = 0, j\neq 3\Big), 
  \end{align*}
  which shows '$\leq$' in \eqref{pfinalfixresult2} (using
  \eqref{eq:816} at time $\tau \frac{\log\alpha}{\alpha}$). For
  '$\geq$', it suffices to show that for any $\varepsilon>0$,
  uniformly in $\underline x$,
  \begin{align*}
    \mathbb P_{\underline x}\Big( \frac{\alpha}{\log\alpha}S<\tau + \varepsilon\Big) 
    &                             =
      \mathbb P_{\underline x}\Big(J_1\Big((\tau+\varepsilon) \frac{\log\alpha}{\alpha}\Big) 
      = J_2\Big((\tau+\varepsilon) \frac{\log\alpha}{\alpha}\Big) = \cdots = 3\Big) 
    \\ & \geq 
         \mathbb P_{\underline x}\Big(\widetilde 
         L_j\Big(\tau\frac{\log\alpha}{\alpha}\Big) = 0, j\neq 3\Big) - \varepsilon
  \end{align*}
  as $\alpha\to\infty$. Recall that in the second term, the random
  variables $J_1, J_2,...$ are types picked at time $\beta = 0$,
  (i.e.\ at $t = (\tau+\varepsilon) \frac{\log\alpha}{\alpha}$) from
  $\mathcal B_0^\infty$. Now, for {\color{black}
    $N := \lfloor \alpha\rfloor$, we have that}
  $\mathbb P(N\leq \Pi)\geq 1-\varepsilon$, define
  $K_1=K_1\Big(\tau \frac{\log\alpha}{\alpha}\Big),...,
  K_N=K_N\Big(\tau \frac{\log\alpha}{\alpha}\Big)$,
  which are types of specific lines at time
  $\beta = \varepsilon \frac{\log\alpha}{\alpha}$ (i.e.\
  $t = \tau \frac{\log\alpha}{\alpha}$) with the property
  \begin{align*}
    \big\{K_1 & = \cdots = K_N=3\big\} \subseteq \Big\{J_1\Big((\tau
                + \varepsilon) \frac{\log\alpha}{\alpha}\Big) = J_2\Big((\tau +
                \varepsilon) \frac{\log\alpha}{\alpha}\Big) = \cdots = 3\Big\}
  \end{align*}
  on the event $\{N\leq \Pi\}$. These
  {\color{black}$N = \lfloor\alpha\rfloor$} lines exist with high
  probability for large $\alpha$, since following incoming lines at
  selective, and both lines at recombination branching events along
  the ASRG between times $\tau \frac{\log\alpha}{\alpha}$ and
  $(\tau + \varepsilon) \frac{\log\alpha}{\alpha}$, we find that
  coloring these lines with type~3 always leads to the event
  $J_1\Big((\tau + \varepsilon) \frac{\log\alpha}{\alpha}\Big) =
  J_2\Big((\tau + \varepsilon) \frac{\log\alpha}{\alpha}\Big) = \cdots
  = 3$.
  \sloppy Since coalescence and recombination branching events bring
  the infinitely many lines down to less than $\alpha$ lines
  in a time $\mathcal O(1/\alpha)$ (see~e.g.  Proposition~6.9 of
  \cite{DGP12}), we have found the specific
  {\color{black}$N = \lfloor\alpha\rfloor$} lines with types
  $K_1,\cdots, K_N$.  Then, from Proposition~\ref{P:fixL3}, and since
  $K_1,...,K_N,J_1\big(\tau\frac{\log\alpha}{\alpha}\big),
  J_2\big(\tau\frac{\log\alpha}{\alpha}\big),...$ are exchangeable,
  \begin{align*}
    \mathbb P_{\underline x}\Big( & \widetilde L_j\Big(\tau\frac{\log\alpha}{\alpha}\Big) = 0, j\neq 3\Big) 
                                    = 
                                    \mathbb P_{\underline x}
                                    \Big(J_1\Big(\tau \frac{\log\alpha}{\alpha}\Big) = \cdots = J_\Pi\Big(\tau 
                                    \frac{\log\alpha}{\alpha}\Big) = 3\Big) 
    \\ & 
         \leq  
         \mathbb P_{\underline x}\Big(K_1\Big(\tau \frac{\log\alpha}{\alpha}\Big) = \cdots = K_N\Big(\tau \frac{\log\alpha}{\alpha}\Big) = 3, N\leq\Pi\Big) + \mathbb P(N>\Pi) 
    \\ & 
         \leq 
         \mathbb P_{\underline x}\Big(J_1\Big((\tau+\varepsilon) \frac{\log\alpha}{\alpha}\Big) 
         = J_2\Big((\tau+\varepsilon) \frac{\log\alpha}{\alpha}\Big) = \cdots = 3\Big) + \varepsilon.
  \end{align*}
\end{proof}

\section{Auxiliary results}
\label{S:aux}
From Proposition~\ref{P:fixProbL}, we have formulated both theorems in
terms of the Markov jump process $\mathcal L$. For this process, we
will need to bound the probabilities of several events in the proof of
the Theorems. We collect basic facts about birth-death processes in
this section.

\subsection{Results on homogeneous birth-death processes}
We recall and extend some results from previous work on similar
models. Above all, we will refer to \cite{greven2014fixation} and
\cite{Cuthetal2012}. We start by recalling a result on the
concentration of the total number of particles in the ASRG at any
time. For this, the following lemma is a special case of Lemma~4.1 of
\cite{greven2014fixation} (using $\rho=1$ and $d=1$ in their
paper). In order to see this, we note that $L_0 + L_1 + L_2 + L_3$
with $L_0,...,L_3$ from Definition~\ref{def:L0} is a
birth-death-process with birth rate $\lambda_k = \alpha k$ and death
rate $\mu_k = (1+\rho/\alpha)\binom k 2$.

\begin{lemma}[$\underline L$ concentrated around
  $2\alpha$\label{l:conc}]
  Let $\mathcal L = (\underline L(t))_{t\geq 0}$ be as in
  Definition~\ref{def:L0} 
  %(Recall that this process depends on the
  %parameters $\alpha$ and $\rho$.)  
  If
  $(L_0(0) + L_1(0) + L_2(0) + L_3(0))/\alpha
  \xRightarrow{\alpha\to\infty} 2$,
  then for any $t_\alpha \downarrow 0$,
  \begin{align*}
    \sup_{0\leq r\leq t_\alpha} \Big| \frac{L_0(r) + L_1(r) + L_2(r) + L_3(r)}{\alpha} - 2\Big|
    \xRightarrow{\alpha\to\infty} 0.
  \end{align*}
\end{lemma}

\noindent
We now state some results on birth-death processes, which can be
approximated by branching processes. The first is Lemma 6.1 from
\cite{Cuthetal2012}.

\begin{lemma}[Binary branching process\label{l2}]
  Let $(L(t))_{t\geq 0}$ be a binary branching process with
  (individual) birth rate $\lambda$ and death rate $\mu$. If $L(0)=1$
  and $\lambda > \mu$, then, for some constant $C_{\lambda,\mu}$,
  which only depends on $\lambda/(\lambda+\mu)$ and
  $\mu/(\lambda+\mu)$,
  \begin{align*}
    |\mathbb P(L(t)=0) - \mu/\lambda| & \leq \mu e^{-(\lambda - \mu)t}/\lambda,\\
    \mathbb P(1\leq L(t)\leq K) & \leq C_{\lambda,\mu} K e^{-(\lambda-\mu)t} \text{ if }K\leq e^{(\lambda-\mu)t}/6,\\
    \mathbb P(\sup_{s\leq t}L(s)\geq K) & \leq C_{\lambda,\mu}e^{(\lambda-\mu)t}/K
  \end{align*}
\end{lemma}

\begin{corollary}[Exponential growth\label{c3}]
  Let $(L(t))_{t\geq 0}$ be a binary branching process with
  (individual) birth rate $\alpha$ and death rate
  $c\alpha + o(\alpha)$ for $0<c<1$ for large $\alpha$. If $L(0)=1$,
  then for $t>0$,
  $$ Y(t) := \frac{\log L(t (\log\alpha)/\alpha)}{\log\alpha} \xRightarrow{\alpha\to\infty} M,$$
  where $\mathbb P(M=-\infty) = c, \mathbb P(M=(1-c)t)=1-c$.
\end{corollary}

\begin{proof}
  By continuity, we only need to show the result for death rate equal
  to $c\alpha$. We use Lemma~\ref{l2} with $\lambda=\alpha$ and
  $\mu=c\alpha$. We get, for any $\delta>0$
  \begin{align*}
    \mathbb P(Y(t) & = -\infty) 
                     = \mathbb P(L(t)=0) 
                     = c + \mathcal O\Big(\frac{c\alpha}{\alpha}e^{-\alpha(1-c)t(\log\alpha)/\alpha}\Big) 
                     \stackrel{\alpha\to\infty}\approx c,\\
    \mathbb P(-\infty & < Y(t) \leq (1-c-\delta)t) 
                        = \mathbb P(1\leq L(t (\log\alpha)/\alpha) 
                        \leq e^{\alpha (1 - c - \delta)t (\log\alpha)/\alpha})
    \\ & \leq C_{1, c}e^{\alpha (1 - c - \delta)t (\log\alpha)/\alpha} 
         e^{-\alpha(1-c)t(\log\alpha)/\alpha} = C_{1, c}e^{-\delta t \log\alpha} 
         \stackrel{\alpha\to\infty}{\approx} 0,
    \\
    \mathbb P(Y(t) & > (1-c+\delta)t) 
                     = 
                     \mathbb P(L(t (\log\alpha)/\alpha) 
                     > e^{\alpha (1 - c + \delta)t (\log\alpha)/\alpha})
    \\ & \leq C_{1,c}e^{\alpha(1-c)t(\log\alpha)/\alpha}e^{-\alpha(1-c+\delta)t(\log\alpha)/\alpha} = 
         C_{1, c}e^{-\delta t \log\alpha} \stackrel{\alpha\to\infty}{\approx} 0.
  \end{align*}
  Hence the result follows.
\end{proof}

~ The following two lemmata are extensions of Lemma 4.6 in
\cite{greven2014fixation}.

\begin{lemma}[Growth rate of binary branching
  process\label{l:growthBin}]
  Let $(L(t))_{t\geq 0}$ be a binary branching process with
  (individual) birth rate $\alpha$ and death rate
  $c\alpha + o(\alpha)$ for $c\geq 0$. If $L(0)=c'\alpha^\psi$ for
  some $\psi,c'>0$ and $Z(t) := \psi + (1-c)t$, then
  $$ Y := \Big(\frac{\log L(t (\log\alpha)/\alpha)}{\log\alpha}\Big)_{t\geq 0} 
  \xRightarrow{\alpha\to\infty} Z.$$
\end{lemma}
  
\begin{proof}
  Again, we only prove the result for death rate $c\alpha$. We have
  that $L(t(\log\alpha)/(\alpha)) = \alpha^{Y(t)}$. Clearly,
  $Y(0) \stackrel{\alpha\to\infty}\approx\psi$, and $Y$ is a Markov
  process with generator, for $f\in\mathcal C^1(\mathbb R)$
  \begin{align*}
    Gf(y) & = \alpha^y \log\alpha \Big(f\Big(\frac{\log(\alpha^y + 1)}{\log\alpha}\Big) - f(y) + c\Big(
            f\Big(\frac{\log(\alpha^y - 1)}{\log\alpha}\Big) - f(y)\Big)\Big) 
    \\ & = \alpha^y \log\alpha \Big(f\Big(y + \frac{\log(1 + \alpha^{-y})}{\log\alpha}\Big) - f(y) + c\Big(
         f\Big(y + \frac{\log(1 - \alpha^{-y})}{\log\alpha}\Big) - f(y)\Big)\Big) 
    \\ & \stackrel{\alpha\to\infty}\approx 
         \alpha^y \log\alpha \Big(f\Big(y + \frac{1}{\alpha^{y}\log\alpha}\Big) - f(y) + c\Big(
         f\Big(y - \frac{1}{\alpha^{y}\log\alpha}\Big) - f(y)\Big)\Big) 
    \\ & \stackrel{\alpha\to\infty}\approx (1-c)f'(y).
  \end{align*}
  By standard arguments (see e.g.\ \citealp[{\color{black} Lemma 4.5.1
    and Remark 4.5.2}]{EthierKurtz1986}), the result follows.
\end{proof}

\begin{lemma}[Hitting times of super-critical branching
  process\label{l:5a}]
  Let $z, c'>0, c\in(0,1), \varepsilon \in (0,c/2)$, and
  ${\mathcal L} = (L_t)_{t\geq 0}$ be a birth-death process with birth
  rate ${b_k} = \alpha k$ and death rate
  $d_k \leq \alpha(1-c + \varepsilon) k$, started in
  $L_0 = z\alpha^{1-\varepsilon}$. Moreover, let $T_{n}$ be the first
  time when $L_t=n$. Then,
  $$ \mathbb P\Big( \frac{\alpha}{\log\alpha}T_{\varepsilon\alpha} > 
  \frac{2\varepsilon}{c} \Big| L_0 = z\alpha^{1-\varepsilon} \Big)
  \xrightarrow{\alpha\to\infty} 0.$$
\end{lemma}

\begin{proof}
  It suffices to treat the case $d_k = \alpha(1-c+\varepsilon)$, since
  other cases have an earlier hitting time of
  $\varepsilon\alpha$. Now, $\mathcal L$ is a supercritical branching
  process and we let $\mathcal M = (M_t)_{t\geq 0}$ be the process of
  immortal lines in $\mathcal L$. Then, by classical theory, we find
  that each line in $\mathcal L$ belongs to $\mathcal M$ with
  probability $c-\varepsilon$. Moreover, for some {\color{black} random
    $Z'>0$ with $Z' \sim B(z\alpha^{1-\varepsilon}, c-\varepsilon)$,
    $M_0 = Z'\alpha^{1-\varepsilon}$, therefore
    $Z'/\alpha^{1-\varepsilon} \xrightarrow{\alpha\to\infty}
    z(c-\varepsilon)$}
  and $\mathcal M$ is a pure branching process with splitting rate
  $(c-\varepsilon)\alpha$. If $S_{\varepsilon\alpha}$ is the first
  time $t$ when $M_t = \varepsilon\alpha$, we find by
  Lemma~\ref{l:growthBin} that
  $$ \frac{\alpha}{\log\alpha} S_{\varepsilon\alpha} \xRightarrow{\alpha\to\infty}
  \frac{\varepsilon}{c-\varepsilon} < \frac{2\varepsilon}{c},$$
  since $\log(L_{t(\log\alpha)/\alpha})/(\log\alpha)$ grows (for large
  $\alpha$) linearly at speed $c-\varepsilon$. Using
  $S_{\varepsilon\alpha}\geq T_{\varepsilon\alpha}$, we conclude by
  \begin{align*}
    \mathbb P\Big( \frac{\alpha}{\log\alpha}T_{\varepsilon\alpha} > 
    \frac{2\varepsilon}{c} \Big| L_0 = z\alpha^{1-\varepsilon} \Big)
    \leq \mathbb P\Big( \frac{\alpha}{\log\alpha}S_{\varepsilon\alpha} > 
    \frac{2\varepsilon}{c} \Big| M_0 = z'\alpha^{1-\varepsilon} \Big)
    \xrightarrow{\alpha\to\infty} 0.
  \end{align*}
\end{proof}

\begin{lemma}[Fast middle phase of local sweep\label{l:3}] Let
  $\varepsilon, z>0$ with $0<z<2(1-\varepsilon)$, and
  ${\mathcal L} = (L_t)_{t\geq 0}$ be a birth-death process, started
  in $L_0 = z\alpha$,
  $$ \text{birth rate ${b_k} = \alpha k$ and death rate ${d_k} = \binom k 2 + \tfrac 12 k(2\alpha - k)c + o(\alpha)$}$$
  for some $c\in[0,1)$. Moreover, let $T_n$ be the first time when
  $L_t=n$. Then,
  $$ 
  T_{2\alpha (1-\varepsilon)} = O\Big(\frac 1\alpha\Big)
  $$
  as $\alpha\to\infty$. In particular, we find a sequence
  $\varepsilon_\alpha\downarrow 0$ such that
  \begin{align}
    \label{eq:l32}
    T_{2\alpha (1-\varepsilon_\alpha)} - T_{2\alpha
    \varepsilon_\alpha} = o\Big(\frac{\log\alpha}\alpha\Big).
  \end{align}
\end{lemma}

\begin{proof}
  We rescale state and space by setting
  $\mathcal V^\alpha :=(V^\alpha_t)_{t\geq 0} :=
  (L_{t/\alpha}/\alpha)_{t\geq 0}$
  and obtain that
  $\mathcal V^\alpha \xRightarrow{\alpha\to\infty} \mathcal V$ with
  $\mathcal V = (V_t)_{t\geq 0}$ solving
  \begin{align*}
    \frac{dV}{dt} & = V(1- \tfrac V2) - V(1- \tfrac V2) c = V(1- \tfrac V2)(1-c),
  \end{align*}
  and starting in $V_0 = z$. {\color{black} (Here, we use again
    \citealp[Lemma 4.5.1 and Remark 4.5.2]{EthierKurtz1986}}) Since
  this process hits $1-\varepsilon$ by some time of order~1, we find
  that $L$ hits $2\alpha(1-\varepsilon)$ by some time $O(1/\alpha)$.
\end{proof}

\begin{lemma}[Hitting times of sub-critical branching
  process\label{l:5}]
  Let $z, c', c>0$, $\gamma\in[0,1), p\in (\gamma,1]$,
  $\varepsilon \in (0,c^2/p \wedge c)$, and
  ${\mathcal L} = (L_t)_{t\geq 0}$ be a birth-death process with birth
  rate ${b_k} = \alpha k$ and death rate ${d_k}$ such that
  $ |d_k - \alpha(1+c) k| \leq \varepsilon\alpha k$, started in
  $L_0 = z\alpha^p$. Moreover, let $T_{n}$ be the first time when
  $L_t=n$. Then,
  $$ \mathbb P\Big(  \Big|\frac{\alpha}{\log\alpha}T_{c'\alpha^\gamma} - \frac{p-\gamma}{c}
  \Big| > 4\varepsilon \Big| L_0 = z\alpha^p \Big)
  \xrightarrow{\alpha\to\infty} 0$$ and
  \begin{align}
    \label{eq:branTS1}
    \mathbb P\Big(  \Big|\frac{\alpha}{\log\alpha}T_0 - \frac{p}{c}
    \Big| > 2\varepsilon\Big| L_0 = z\alpha^p \Big)
    \xrightarrow{\alpha\to\infty} 0.
  \end{align}
\end{lemma}

\begin{proof}
  It suffices to show \eqref{eq:branTS1}, since the hitting time of~0
  is (by the Markov property) the independent sum of the hitting time
  of $c'\alpha^\gamma$ and the extinction time, if the process is
  started in $c'\alpha^\gamma$.

  Define $S_{m}^c$ be the extinction time of a branching process
  $\mathcal M = (M_t)_{t\geq 0}$ with (individual) branching rate
  $\alpha$ and death rate $\alpha(1+c)$, when started in $M_0=m$.
  Then, from classical theory (see e.g.\
  \cite[Chapter~V~(3.4)]{Harris1963}) it follows, that
  \begin{align*}
    f(t) & :=\mathbb P(S_1^c>t) =\frac{c}{(1+c) e^{t\alpha c} -
           1},\\
    g_m(t) & := \mathbb P(S^c_m>t) = 1 - (1-f(t))^m.
  \end{align*}
  Hence, for any $\varepsilon>0$,
  \begin{equation}
    \label{eq:2198}
    \begin{aligned}
      \mathbb{P} \Bigg( \frac{\alpha} {\log \alpha} S^c_{z\alpha^p} -
      \frac{p}{c} > \varepsilon \Bigg) & = g_{z\alpha^p} \Bigg(
      \frac{\log \alpha}{ \alpha} \Bigg( \frac{p}{c} + \varepsilon
      \Bigg) \Bigg)
      \rightarrow 0,\\
      \mathbb{P} \Bigg( \frac{\alpha} {\log \alpha} S^c_{z\alpha^p} -
      \frac{p}{c} < - \varepsilon \Bigg) & = 1- g_{z\alpha^p} \Bigg(
      \frac{ \log \alpha}{ \alpha} \Bigg( \frac{p}{c} - \varepsilon
      \Bigg) \Bigg) \rightarrow 0.
    \end{aligned}
  \end{equation}
  Stochastically,
  $S_{z\alpha^p}^{c+\varepsilon} \leq T_{z\alpha^p} \leq
  S_{z\alpha^p}^{c-\varepsilon}$
  and hence, for $\varepsilon\leq c^2/p$
  \begin{align*}
    \mathbb P\Big( \frac{\alpha}{\log\alpha}T_{z\alpha^p} - \frac pc < -
    2\varepsilon\Big) 
    & \leq \mathbb P\Big(
      \frac{\alpha}{\log\alpha}S^{c+\varepsilon}_{z\alpha^p} -
      \frac p{c+\varepsilon}< -
      {{2}\varepsilon +
      \frac{p\varepsilon/c}{c+\varepsilon}}\Big)
      \xrightarrow{\alpha\to\infty} 0 
      \intertext{as well as} 
      \mathbb P\Big( \frac{\alpha}{\log\alpha}T_{z\alpha^p} - \frac pc >
      {2}\varepsilon\Big) 
    & \leq \mathbb P\Big(
      \frac{\alpha}{\log\alpha}S^{c-\varepsilon}_{z\alpha^p} -
      {\frac{p}{c-\varepsilon}} >
      {2\varepsilon -
      \frac{p\varepsilon/c}{c-\varepsilon}}\Big)
      \xrightarrow{\alpha\to\infty} 0
  \end{align*}
  by~\eqref{eq:2198} and we are done.
\end{proof}

\subsection{Results on non-homogeneous birth-death processes}
\cite{Kendall1948} studied time-inhomogeneous birth-death processes;
see also \cite{Bailey1990} and \cite{Allen2013}. We recall a result on
the generating function of such a process. In order to be
self-contained, we also give the proof.

\begin{apptheorem}[Time-inhomogeneous birth-death
  process\label{T:tibd}]
  Let $\mathcal L = (L_t)_{t\geq 0}$ be a birth-death process with
  time-inhomogeneous birth rates
  $\lambda_k(t) = \lambda_t k + \gamma_t$ {\color{black} for $k\geq 0$
    and $\lambda_k(t)=0$ for $k<0$}, and death rates
  $\mu_k(t) = \mu_t k${\color{black}, $k=0,1,2,...$} If $L_0 = \ell$,
  the generating function
  $g_t(z) := \sum_{k=0}^\infty \mathbb P(L_t=k) z^k$ satisfies
  \begin{equation}
    \label{eq:gtz}
    \begin{aligned}
      g_t(z) & = \left(1 - \left(\frac{\exp\Big(\int_{0}^t (\mu_r-
            \lambda_r) {\color{black} dr}\Big)}{1-z} +
          \int_{0}^t \lambda_r \exp\Big(\int_{0}^r (\mu_u - \lambda_u)
          du\Big)dr\right)^{-1}\right)^\ell \\ & \quad \cdot
      \exp\left( - \int_0^t \gamma_{s} \left(\frac{\exp\Big(\int_{s}^t
            (\mu_r- \lambda_r) dr\Big)}{1-z} + \int_{s}^t \lambda_r
          \exp\Big(\int_{s}^r (\mu_u - \lambda_u)
          du\Big)dr\right)^{-1} ds \right).
    \end{aligned}
  \end{equation}
  In particular, the probability that the process eventually goes
  extinct, starting with $\ell=0$, is given by
  \begin{align}\label{eq:surv}
    \mathbb P(L_t = 0) & = \exp\left(
                         - \int_0^t \gamma_{s}\left(
                         1 +  \int_{s}^t \mu_r
                         e^{\int_{s}^r (\mu_u - \lambda_u) du}dr
                         \right)^{-1}ds\right).
  \end{align}
\end{apptheorem}

\begin{proof}
  We write, using $p_k(t) := \mathbb P(L_t=k)$,
  \begin{equation}
    \label{eq:char}
    \begin{aligned}
      \frac{\partial g_t(z)}{\partial t} & = \frac{\partial}{\partial
        t} \mathbb E[z^{L_t}]=
      {\color{black}\sum_{k=0}^\infty}\big(((\lambda_t(k-1) + \gamma_t)
      p_t(k-1) + \mu_t(k+1)p_{k+1}(t) \\ & \qquad \qquad \qquad \qquad
      \qquad \qquad \qquad - ((\lambda_t + \mu_t)k +
      \gamma_t)p_k(t)\big)z^k \\ & = \big(-\lambda_t z(1-z) + \mu_t(1
      - z)\big) \frac{\partial g_t(z)}{\partial z} - \gamma_t(1-z)
      g_t(z).
    \end{aligned}    
  \end{equation}
  Therefore, if $x$ solves the ODE (with $s\in[0,t]$)
  \begin{align}
    \label{eq:ODE}
    \frac{d}{ds} x_s = -\lambda_{t-s} x_s(1-x_s) + \mu_{t-s}(1 - x_s) \text{   with $x_0=z$,}
  \end{align}
  we find that
  \begin{align*}
    \frac{d}{ds} & \mathbb E\Big[x_{t-s}^{L_s} \exp\Big(-\int_0^{t-s} \gamma_{t-u}(1-x_{u})du\Big)\Big] 
    \\ & = \mathbb E\Big[\Big(\big(-\lambda_s x_{t-s}(1-x_{t-s}) +
         \mu_s(1 - x_{t-s})\big) \frac{\partial}{\partial z} z^{L_s}\Big|_{z=x_{t-s}}
         - \gamma_s(1-x_{t-s}) x_{t-s}^{L_s}
    \\ & \qquad - \frac{\partial}{\partial z} z^{L_s}\Big|_{z=x_{t-s}} (-\lambda_s x_{t-s}(1-x_{t-s}) + \mu_s(1-x_{t-s})) 
    \\ & \qquad \qquad \qquad \qquad \qquad \qquad \qquad + x_{t-s}^{L_s} \gamma_{s}(1-x_{t-s})\Big)
         \exp\Big(-\int_0^{t-s} \gamma_{t-u}(1-x_u)du\Big)\Big]
    \\ & = 0.
  \end{align*}
  Thus, the solution of~\eqref{eq:char} is given by
  \begin{align*}
    g_t(z) & = g_0(x_t) \exp\Big(-\int_0^t \gamma_{t-s}(1-x_s)ds\Big).
  \end{align*}
  Since 
  \begin{align*}
    x_s & = 1 - \left(\frac{\exp\Big(\int_{t-s}^t (\mu_r- \lambda_r) {\color{black}dr}\Big)}{1-z} + 
          \int_{t-s}^t \lambda_r \exp\Big(\int_{t-s}^r (\mu_u - \lambda_u) du\Big)dr\right)^{-1}
  \end{align*}
  solves~\eqref{eq:ODE}, we find~\eqref{eq:gtz}. Then, \eqref{eq:surv}
  arises for $z=0$, since
  \begin{equation}
    \label{eq:4321}
    \begin{aligned}
      e^{\int_{t-s}^t (\mu_r- \lambda_r) dr} & + \int_{t-s}^t
      \lambda_r e^{\int_{t-s}^r (\mu_u - \lambda_u) du}dr \\ & =
      e^{\int_{t-s}^t (\mu_r- \lambda_r) dr} + \int_{t-s}^t \mu_r
      e^{\int_{t-s}^r (\mu_u - \lambda_u) du}dr -\int_{t-s}^t (\mu_r -
      \lambda_r) e^{\int_{t-s}^r (\mu_u - \lambda_u) du} dr \\ & = 1 +
      \int_{t-s}^t \mu_r e^{\int_{t-s}^r (\mu_u - \lambda_u) du}dr.
    \end{aligned}
  \end{equation}
\end{proof}

\begin{proposition}[Survival of a non-homogeneous branching process
  with immigration\label{P:49}]
  Let $0<c_1<c_2<1$ and $\rho>0$ and let $y_t$ solve
  $dy = (c_2-c_1)y(1-y)$ with $y_0=1/2$. Let
  $\mathcal L = (L(t))_{t\in\mathbb R}$ be a binary non-homogeneous
  branching process with $L_{-\infty} =0$, such that every individual
  gives birth at rate $\lambda = 1$, dies at rate
  \begin{align*}
    \mu_t = c_1 (1-y_t) + c_2 y_t = y_t(c_2-c_1) + c_1,    
  \end{align*}
  and immmigration rate $\gamma_t = \rho y_t(1-y_t)$. Then, the
  probability that the process survives is
  \begin{align*}
    \mathbb P[L_\infty > 0] & = 1 - \Big(\frac{1-c_1}{1-c_2}\Big)^{-\rho\frac{(1-c_1)(1-c_2)}{(c_2-c_1)^2}}.
  \end{align*}
\end{proposition}

\begin{proof}
  {\color{black}We directly use \eqref{eq:gtz} with $\ell=0$ and $z=0$
    from Theorem ~\ref{T:tibd} and obtain} 
  \begin{align*}
    - \log & \mathbb P(L_\infty = 0) = \lim_{t\to\infty}\int_{-t}^t 
             \rho y_s(1-y_s) 
    \\ & \qquad \qquad \qquad \qquad \cdot \Big(
         \exp\Big(-\int_s^t (1-c_1 - y_r(c_2-c_1))dr\Big) 
    \\ & \qquad \qquad \qquad \qquad \qquad + \int_s^t
         \exp\Big(-\int_s^r  (1-c_1 - y_u(c_2-c_1))du\Big)dr\Big)^{-1} ds\Big).    
  \end{align*}
  Since
  \begin{align*}
    \int_{s}^r (1-c_1 - y_u(c_2-c_1)) du 
    & = (1-c_1)(r-s) - \int_{y_s}^{y_r} \frac{1}{1-z}dz 
    \\ & = (1-c_1)(r-s) + \log\Big(\frac{1-y_r}{1-y_s}\Big)
  \end{align*}
  and $y_r = \frac{1}{1 + e^{-(c_2-c_1)r}}$, we find
  \begin{align*}
    \int_{s}^t & \exp\Big(-\int_{s}^r (1-c_1-y_u(c_2-c_1) )du\Big)dr
                 = (1-y_s)\int_{s}^t e^{-(1-c_1)(r-s)} \frac{1}{1-y_r}dr
    \\ & = (1-y_s) e^{(1-c_1)s}\int_{s}^t e^{-(1-c_1)r}\frac{1 + e^{-(c_2-c_1)r}}{e^{-(c_2-c_1)r}}dr
    \\ & =  (1-y_s) e^{(1-c_1)s}\int_{s}^t e^{-(1-c_2)r} + e^{-(1-c_1)r}dr
    \\ & =  (1-y_s) e^{(1-c_1)s}\Big(\frac{1}{1-c_2} (e^{-(1-c_2)s} -e^{-(1-c_2)t}) 
         + \frac{1}{1-c_1}((e^{-(1-c_1)s} -e^{-(1-c_1)t}) )\Big)
    \\ & \xrightarrow{t\to\infty}(1-y_s)\Big(\frac{1}{1-c_2} e^{(c_2-c_1)s} 
         + \frac{1}{1-c_1}\Big) = \frac{1}{1-c_2}y_s + \frac{1}{1-c_1}(1-y_s).
  \end{align*}
  Therefore, 
  \begin{align*}
    - \log & \mathbb P(L_\infty = 0) = \lim_{t\to\infty}\int_{-t}^t \rho y_s(1-y_s) 
    \\ & \qquad \qquad \qquad \qquad \cdot      
         \Big(\Big(\frac{1-y_s}{1-y_t}\Big)^{- (1-c_1)(t-s)} 
         + \frac{1}{1-c_2}y_s  + \frac{1}{1-c_1}(1-y_s)\Big)^{-1}ds
    \\ & = \int_{-\infty}^\infty \rho y_s(1-y_s) 
         \frac{(1-c_1)(1-c_2)}{(1-c_1)y_s + (1-c_2)(1-y_s)}ds
    \\ & = \rho\frac{(1-c_1)(1-c_2)}{c_2-c_1} \int_0^1 \frac{1}{(1-c_2) + (c_2-c_1)y}dy
    \\ & = \rho\frac{(1-c_1)(1-c_2)}{(c_2-c_1)^2} \log\frac{1-c_1}{1-c_2}
  \end{align*}
  and we are done.
\end{proof}

\subsection{Some bounds on $\mathcal L$}
Frequently, we will use the stopping times
\begin{align}\label{eq:stoptimes}
  T_k^i := \inf\{t\geq 0: L_i(t) \geq k\}.
\end{align}

\begin{lemma}[Type~3 does not occur before type~2 reaches $\varepsilon\alpha$\label{l:10}]
  Let $\mathcal L$ be as in Definition~\ref{def:L0}. Then, for any
  $\varepsilon_\alpha \xrightarrow{\alpha\to\infty} 0$,
  \begin{align}
    \lim_{\alpha\to\infty} \mathbb P(T_1^3 \geq T_{\varepsilon_\alpha\alpha}^2) = 1.
  \end{align}
\end{lemma}

\begin{proof}
  Type-3-particles only arise by an event of rate
  $\ell_1 \ell_2 \rho/(2\alpha)$, so the time $T_1^3$ is the same if
  we ignore the decrease in type-2-particles due to this event. In
  other words, we consider the rates of in- and decrease of
  type-2-particles, given that $\ell_2 \leq \varepsilon_\alpha \alpha$
  and $\ell_0 + \ell_1 + \ell_2 \leq 2\alpha(1+\delta)$ (see
  Lemma~\ref{l:conc})
  \begin{align*}
    r_2^+ & := \alpha \ell_2,\\
    r_2^- & := \binom{\ell_2}{2} 
            + \frac 12 \ell_2 \ell_1 (1 - c_2 +
            c_1) + \frac 12 \ell_2 \ell_0 (1 - c_2)
    \\ & \leq \frac 12 \ell_2\big( (1 - c_2 + c_1)(\ell_0 + \ell_1 + \ell_2) 
         + (c_2 - c_1)\varepsilon_\alpha \alpha\big) 
    \\ & \stackrel{\alpha\to\infty}\lesssim (1 - c_2 + c_1)(1+\delta)\alpha \ell_2.
  \end{align*}
  Hence, using Lemma~\ref{l:conc} and Lemma~\ref{l:9}, and $\delta>0$
  such that $(1-c_2+c_1)(1+\delta) < 1$,
  \begin{align*}
    \mathbb P(T_1^3 \geq T_{\varepsilon_\alpha\alpha}^2) 
    & 
      = \mathbb E\Big[ \exp\Big( - \int_0^{T_{\varepsilon_\alpha \alpha}^2} 
      \frac{\rho}{2\alpha} L_1(t) L_2(t) dt\Big)\Big] 
    \\ & \geq \exp\Big( -  \mathbb E\Big[\int_0^{T_{\varepsilon_\alpha\alpha}^2} 
         \frac{\rho}{2\alpha} L_1(t) L_2(t) dt\Big]\Big)
    \\ & \geq \exp\Big( - \mathbb E\Big[\int_0^{T_{\varepsilon_\alpha\alpha}^2} 
         \rho (1+\delta) L_2(t) dt\Big]\Big)
    \\ & \geq  \exp\Big( - \frac{\rho (1+\delta)}{\alpha(1-(1-c_2+c_1)(1+\delta))}
         \varepsilon_\alpha\alpha\Big) \stackrel{\alpha\to\infty}\approx 1.
  \end{align*}
\end{proof}

\begin{lemma}[Occupation time of birth-death process\label{l:9}]
  Let $(L(t))_{t\geq 0}$ be a birth-death-process with birth rate
  $\lambda_i \geq ai$ and death rate $\mu_i \leq bi$. If $a>b$ and
  $L(0) = k$, for $T_\ell := \inf\{t: L(t)=\ell\}$ and $\ell>k$,
  $$ \mathbb E\Big[ \int_0^{T_\ell} L(s)ds\Big] \leq \frac{\ell - k}{a-b}.$$
\end{lemma}

\begin{proof}
  We know that
  $$ \Big(L(t) - L(0) - \int_0^t \lambda_{L(s)} - \mu_{L(s)}  ds \Big)_{t\geq 0}$$
  is a (local) martingale, hence,
  \begin{align}
    \label{eq:localSub}
    \Big(L(t) - L(0) - \int_0^t (a-b) L(s)  ds \Big)_{t\geq 0}
  \end{align}
  is a (local) sub-martingale. Using optional stopping {\color{black}
    and the fact that the process in \eqref{eq:localSub}, stopped at
    $T_\ell$, is a true sub-martingale},
  \begin{align*}
    \mathbb E\Big[\int_0^{T_\ell} (a-b) L(s)  ds\Big]
    & \leq \mathbb E[L(T_\ell) - L(0)] 
      = (\ell - k) \mathbb P(T_\ell<\infty) - k\cdot \mathbb P(T_\ell=\infty) \\ & \leq \ell - k.
  \end{align*}
\end{proof}

\begin{proposition}[Rescaling $\mathcal L$ when type~2 takes over\label{P:rescale}]
  Let $\mathcal L$ from Definition~\ref{def:L0} be started with
  \begin{align*}
    L_0(0) & = o(\alpha), \qquad 
             L_1(0) = 2(1-\varepsilon)\alpha + o(\alpha), \qquad 
             L_2(0) = 2\varepsilon\alpha + o(\alpha), \qquad 
             L_3(0)=0.
  \end{align*}
  Moreover, let 
  $$V^\alpha_i(t) := \frac{L_i(t/\alpha)}{\alpha}, i=0,1,2, \qquad V_3^\alpha(t) :=
  L_3(t/\alpha)$$
  and $V = (V_0, V_1, V_2, V_3)$ be a process with $V_0=0$,
  $V_1(0)=2(1-\varepsilon), V_2(0) = 2\varepsilon$ and $V_3(0)=0$,
  such that
  \begin{align*}
    dV_1 = (c_2-c_1)V_1(1-V_1/2)dt, \qquad V_2 = 2-V_1,
  \end{align*}
  and $V_3$ is a time-dependent binary branching process with
  splitting rate~1, death rate $\frac{c_1}{2}V_2 + \frac{c_2}{2} V_1$,
  and immigration rate $\frac{\rho}{2} V_1V_2$. Then,
  \begin{align*}
    (V_0^\alpha, V_1^\alpha, V_2^\alpha, V_3^\alpha)
    \xRightarrow{\alpha\to\infty} (V_0, V_1, V_2, V_3).
  \end{align*}
\end{proposition}

\begin{proof}
  Recall that $\mathcal L$ can be seen as a chemical reaction network
  as described in Remark~\ref{rem:chem}. For such networks, limit
  results using a scaling parameter ($\alpha$ in our case) have been
  established in \cite{BallKurtzPopovicRempala2006}, \cite{KangKurtz2013},
  \cite{PfaffelhuberPopovic2015} and others. The following is an
  application of Lemma~2.4 of \cite{PfaffelhuberPopovic2015} (see also
  Theorem~4.1 of \cite{KangKurtz2013}). We use the representation of
  the process $\mathcal L$ using time-change equations of the form
  (recall that $c_0=0$)
  \begin{align*}
    V_0^\alpha(t) & = V_0^\alpha(0) + \frac{1}{\alpha} Y_{b0}\Big(\int_0^t \alpha V_0^\alpha(s)ds\Big) - 
                    \frac{1}{\alpha}
                    Y_{d0}\Big(\frac \alpha 2\int_0^t V_0^\alpha(s) \big(V_0^\alpha(s) - \tfrac{1}{\alpha}\big)ds\Big) 
    \\ & - \sum_{j=1,2} \frac{1}{\alpha}
         Y_{dj0}\Big(\frac 12 (1+c_j) \alpha \int_0^t V_0^\alpha(s) V_j^\alpha(s)ds\Big) - \frac{1}{\alpha}
         Y_{d30}\Big(\int_0^t V_0^\alpha(s) V_3^\alpha(s)ds\Big) 
    \\ & +
         \frac 1\alpha Y_{r120} \Big( \int_0^t \frac \rho 2 V_1^\alpha(s)V_2^\alpha(s) ds\Big) 
         - \frac 1\alpha (Y_{r031}+Y_{r032}) \Big(\int_0^t \frac{\rho}{2\alpha} V_0^\alpha(s)V_3^\alpha(s) ds\Big),\\
    V_1^\alpha(t) & = V_1^\alpha(0) + \frac{1}{\alpha} Y_{b1}\Big(\int_0^t \alpha V_1^\alpha(s)ds\Big) -
                    \frac{1}{\alpha}
                    Y_{d1}\Big(\frac \alpha 2\int_0^t V_1^\alpha(s) \big(V_1^\alpha(s) - \tfrac{1}{\alpha}\big)ds\Big) 
    \\ & - \sum_{j=0,2} \frac{1}{\alpha}
         Y_{dj1}\Big(\frac 12 (1-c_1+c_j)\alpha \int_0^t V_1^\alpha(s) V_j^\alpha(s)ds\Big) 
    \\ & \qquad \qquad \qquad \qquad \qquad \qquad - \frac{1}{\alpha}
         Y_{d31}\Big((1-c_1/2)\int_0^t V_1^\alpha(s) V_3^\alpha(s)ds\Big) 
    \\ & +
         \frac 1\alpha Y_{r031} \Big( \int_0^t \frac\rho{2\alpha} V_0^\alpha(s)V_3^\alpha(s) ds\Big) 
         - \frac 1\alpha (Y_{r120} + Y_{r123}) \Big(\int_0^t \frac \rho 2 V_1^\alpha(s)V_2^\alpha(s) ds\Big),\\
    V_2^\alpha(t) & = V_2^\alpha(0) + \frac{1}{\alpha} Y_{b2}\Big(\int_0^t \alpha V_2^\alpha(s)ds\Big) -
                    \frac{1}{\alpha}
                    Y_{d2}\Big(\frac \alpha 2 \int_0^t V_2^\alpha(s) \big(V_2^\alpha(s) - \tfrac{1}{\alpha}\big)ds\Big) 
    \\ & \qquad - \sum_{j=0,1} \frac{1}{\alpha}
         Y_{dj2}\Big(\frac 12 (1-c_2+c_j)\alpha \int_0^t V_2^\alpha(s) V_j^\alpha(s)ds\Big) 
    \\ & \qquad \qquad \qquad \qquad \qquad \qquad - \frac{1}{\alpha}
         Y_{d32}\Big((1-c_2/2)\int_0^t V_2^\alpha(s) V_3^\alpha(s)ds\Big) 
    \\ & +
         \frac 1\alpha Y_{r032} \Big( \int_0^t \frac\rho{2\alpha} V_0^\alpha(s)V_3^\alpha(s) ds\Big) 
         - \frac 1\alpha (Y_{r120} + Y_{r123}) \Big(\int_0^t \frac \rho 2 V_1^\alpha(s)V_2^\alpha(s) ds\Big),\\
%  \end{align*}
%  \begin{align*}
    V_3^\alpha(t) & = V_3^\alpha(0) + Y_{b3}\Big(\int_0^t V_3^\alpha(s)ds\Big) -
                    Y_{d3}\Big(\frac 1 \alpha \int_0^t V_3^\alpha(s) (V_3^\alpha(s) - 1)ds\Big) 
    \\ & \qquad -
         Y_{d13}\Big(\frac{c_1}{2}\int_0^t V_3^\alpha(s) V_1^\alpha(s)ds\Big) 
         - Y_{d23}\Big(\frac{c_2}{2}\int_0^t V_2^\alpha(s) V_3^\alpha(s)ds\Big) 
    \\ & \qquad +
         Y_{r123} \Big( \int_0^t \frac \rho 2 V_1^\alpha(s)V_2^\alpha(s) ds\Big) 
         - (Y_{r031} + Y_{r032}) \Big(\int_0^t \frac\rho{2\alpha} V_0^\alpha(s)V_3^\alpha(s) ds\Big),
  \end{align*}
  where all $Y$'s are independent rate-one Poisson processes.

  From Lemma~\ref{l:conc}, $(V^\alpha_0, V^\alpha_1, V^\alpha_2)$
  satisfy the compact containment condition. Therefore, since
  $V^\alpha_3$ is bounded by a binary branching process with
  immigration~{\color{black}
    $\frac \rho 2 V_1^\alpha V_2^\alpha$}, it satisfies the compact
  containment condition as well. Therefore, neglecting all terms of
  lower order according to Lemma~2.4 of \cite{PfaffelhuberPopovic2015}
  (i.e.\ using that
  $\frac{1}{\alpha} Y(t) \xrightarrow{\alpha\to
    {\color{black}\infty}}0$),
  we can approximate $V^\alpha$ by $V$, which satisfies
  \begin{align*}
    V_0(t) 
    & =  V_0(0) + \int_0^t  V_0(s)ds - 
      \frac 12\int_0^t ( V_0(s))^2
      - \sum_{j=1,2}
      \frac 12 (1+c_j)  \int_0^t  V_0(s)  V_j(s)ds,
    \\
    V_1(t) 
    & =  V_1(0) + \int_0^t   V_1(s)ds -
      \frac 1 2\int_0^t ( V_1(s))^2 - \sum_{j=0,2} 
      \frac 12 (1-c_1+c_j) \int_0^t  V_1(s)  V_j(s)ds,\\
    V_2(t) & =  V_2(0) + \int_0^t  V_2(s)ds -
             \frac 1 2 \int_0^t ( V_2(s))^2 ds 
    \\ & \qquad \qquad \qquad \qquad \qquad \qquad 
         - \sum_{j=0,1}\frac 12 (1-c_2+c_j) \int_0^t  V_2(s)  V_j(s)ds,
    \\ 
    V_3(t) & =  V_3(0) + Y_{b3}\Big(\int_0^t  V_3(s)ds\Big) -
             Y_{d13}\Big(\frac{c_1}{2}\int_0^t  V_3(s)  V_1(s)ds\Big) 
    \\ & \qquad \qquad \qquad \qquad 
         - Y_{d23}\Big(\frac{c_2}{2}\int_0^t  V_2(s)  V_3(s)ds\Big) 
         +         Y_{r123} \Big( \int_0^t \frac \rho 2 V_1(s)  V_2(s) ds\Big).
  \end{align*}
  Since $V^\alpha_0(0) \xrightarrow{\alpha\to\infty}0$, we see that
  $V_0=0$. Consequently, since $V_0 + V_1 + V_2=2$ by
  Lemma~\ref{l:conc}, we see that $V_1 + V_2=2$ and therefore
  $(V_1, V_2)$ satisfy
  \begin{align*}
    dV_1 & = \Big(V_1 - \frac{1}{2}V_1(V_1 + V_2) - \frac 12 (c_2-c_1)V_1V_2\Big) dt = - \frac 12 (c_2-c_1)V_1(2-V_1)dt,\\
    dV_2 & = -dV_1,
  \end{align*}
  and $V_3$ is a branching process with splitting rate~1, (individual)
  death rate $\frac{c_1}{2}V_1 + \frac{c_2}{2}V_2$ and immigration
  rate $\frac{\rho}{2}V_1 V_2$, as claimed.
\end{proof}

\section{Proof of Theorems~\ref{theoremfixprob2} and~\ref{T:fixTime}}
\label{S:proofs}
In the light of Proposition~\ref{P:fixProbL}, we have to show fixation
of type~3 in $L$. In order to do this, we give a fundamental result
(Proposition~\ref{P:fixScenario}) in Subsection~\ref{Sec:41}. Proofs
of both theorems are given in Sections~\ref{Sec:proof1} and
\ref{Sec:proof2}, respectively.

\subsection{ A fundamental result; Proofs of both Theorems}
\label{Sec:41}
The following result is fundamental for the proofs of our main
results. Its proof is given in the next subsection. For illustration,
consult Figures~\ref{fig:compSweep1} and~\ref{fig:compSweep2}.

\begin{proposition}[Scenarions of fixation and $\mathcal L$\label{P:fixScenario}]
  Assume the same situation as in Theorem~\ref{theoremfixprob2} and
  $\mathcal L$ as in Definition~\ref{def:L0}.
  \begin{enumerate}
  \item If $\psi < \frac{c_{1}}{c_{2}}$, let
    \begin{equation}
      \label{eq:taus}
      \begin{aligned}
        p & = 1-\left(\frac{1-c_{2}}{1-c_{1}}\right)^{\frac{2\rho
            (1-c_{2}) (1-c_{1})}{(c_{2}-c_{1})^2}},\\
        \tau_1 & = \frac{\psi}{c_1},\\
        \tau_2 & = \tau_1 + \frac{1}{c_2-c_1}\Big(1- \frac{c_2\psi}{c_1}\Big),\\
        \tau_3 & = \tau_2 + \frac{1}{1-c_2},\\
        \tau_4 & = \tau_3 + \frac{1}{1-c_2},
      \end{aligned}
    \end{equation}
    and $Q = (Q_2, Q_3)$ be a stochastic process as follows. Starting
    with $(0, -\infty)$, the process is with probability $1-c_2$
    \begin{align}\label{P33a}
      Q_2(\tau) =     
      \begin{cases}
        0, & \tau = 0,\\
        -\infty, & \tau >0
      \end{cases} \quad \text{ and }\qquad Q_3(\tau) = -\infty,
    \end{align}
    with probability $c_2(1-p)$
    \begin{align}\label{P33b}
      Q_2(\tau) =     
      \begin{cases}
        c_2\tau , & \tau \leq \tau_1,\\
        \frac{c_2}{c_1}\psi + (c_2-c_1)(\tau - \tau_1), &
        \tau_1 < \tau < \tau_2,\\
        1, & t > \tau_2
      \end{cases} \qquad \text{and}\qquad Q_3(\tau) = -\infty,
    \end{align}
    with probability $c_2 p$
    \begin{align}\label{P33c}
      (Q_2(\tau), Q_3(\tau)) =     
      \begin{cases}
        (c_2\tau, -\infty), & \tau \leq \tau_1,\\
        \Big(\frac{c_2}{c_1}\psi + (c_2-c_1)(\tau - \tau_1),
        -\infty\Big), &
        \tau_1 < \tau \leq \tau_2,\\
        (1, (c_3-c_2)(\tau - \tau_2)), & \tau_2 < \tau \leq \tau_3,\\
        (1 - (1-c_2)(\tau - \tau_3), 1), & \tau_3 < \tau \leq \tau_4,\\
        (-\infty, 1), & \tau > \tau_4.
      \end{cases}
    \end{align}
    Then, (with $\Rightarrow$ denoting
    convergence of finite-dimensional distributions),
    $$ \Bigg(\frac{\log L_2\Big(\tau \frac{\log\alpha}{\alpha}\Big)}{\log\alpha}, 
    \frac{\log L_3\Big(\tau
      \frac{\log\alpha}{\alpha}\Big)}{\log\alpha}\Bigg)_{\tau\geq 0,
      \tau\neq \tau_2} \xRightarrow{\alpha\to\infty} (Q_2(\tau),
    Q_3(\tau))_{\tau\geq 0, \tau\neq \tau_2}.$$
    In the last case \eqref{P33c}, we find that
    $L_j\Big(\tau \frac{\log\alpha}{\alpha}\Big)
    \xRightarrow{\alpha\to\infty} 0$
    for $j=0,1,2$ if and only if $\tau\geq \tau_4$.
  \item If $\frac{c_{1}}{c_{2}} < \psi \leq 1$, let
    \begin{align*}
      \sigma_1 & = \frac{1}{c_2},\\
      \sigma_2 & = \sigma_1 + \frac{1-\psi + c_1/c_2}{c_2 - c_1}.
    \end{align*}
    Let $R = (R_1, R_2, R_3)$ be a stochastic process as
    follows. Starting with $(1-\psi, 0, -\infty)$, the process is with
    probability $1-c_2$
    \begin{align}\label{P33ba}
      R_1(\tau) =     
      \begin{cases}
        1-\psi + c_1\tau, & 0\leq\tau\leq \tau_1,\\
        1, & \tau > \tau_1
      \end{cases} \text{ and } (R_2, R_3) = (Q_2, Q_3) \text{ from~\eqref{P33a}}, 
    \end{align}
    with probability $c_2$
    \begin{align}\notag
      (R_1& (\tau), R_2(\tau), R_3(\tau)) 
      \\ & =     \label{P33ca}
      \begin{cases}
        (1-\psi + c_1\tau, c_2\tau, -\infty), & 0\leq\tau\leq \sigma_1,\\
        (1-\psi + c_1\sigma_1 - (c_2-c_1)(\tau - \sigma_1), 1, -\infty), &
        \sigma_1 < \tau \leq \sigma_2,\\
        (-\infty, 1, -\infty), & \tau > \sigma_2.
      \end{cases}
    \end{align}
    Then,
    $$ \Bigg(\frac{\log L_i\Big(\tau \frac{\log\alpha}{\alpha}\Big)}{\log\alpha}\Bigg)_{i=1,2,3,\tau\geq 0} 
    \xRightarrow{\alpha\to\infty} (R_i(\tau))_{i=1,2,3,\tau\geq 0}.$$
    In particular, $L_3\xRightarrow{\alpha\to\infty} 0$.
  \end{enumerate}
\end{proposition}

\begin{figure}
  \hspace{3cm} 
  \begin{center}
    \begin{tabular}{cc}
      (A) \hspace{1cm} & \parbox{9cm}{\includegraphics[width=9cm]{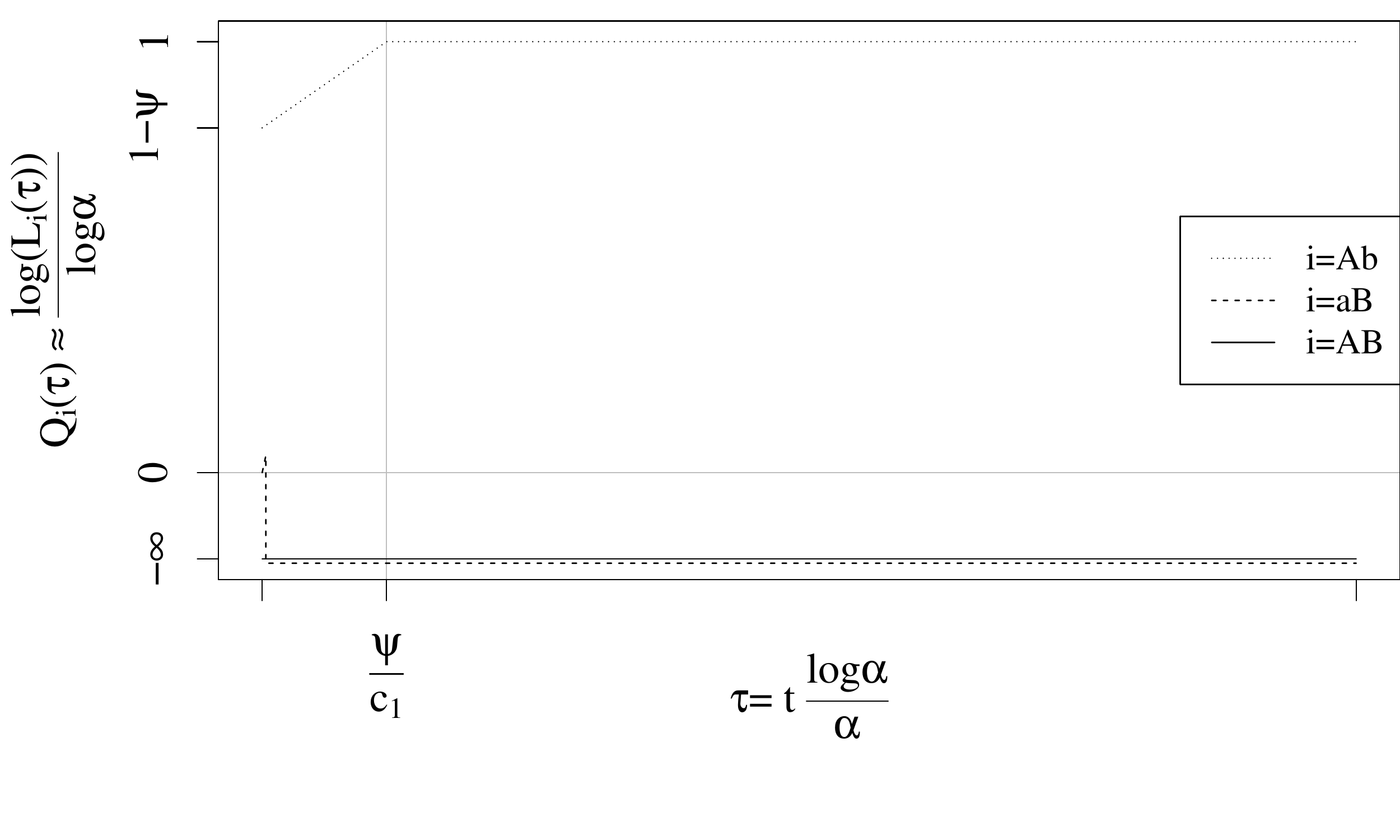}} \\
      (B) \hspace{1cm} & \parbox{9cm}{\includegraphics[width=9cm]{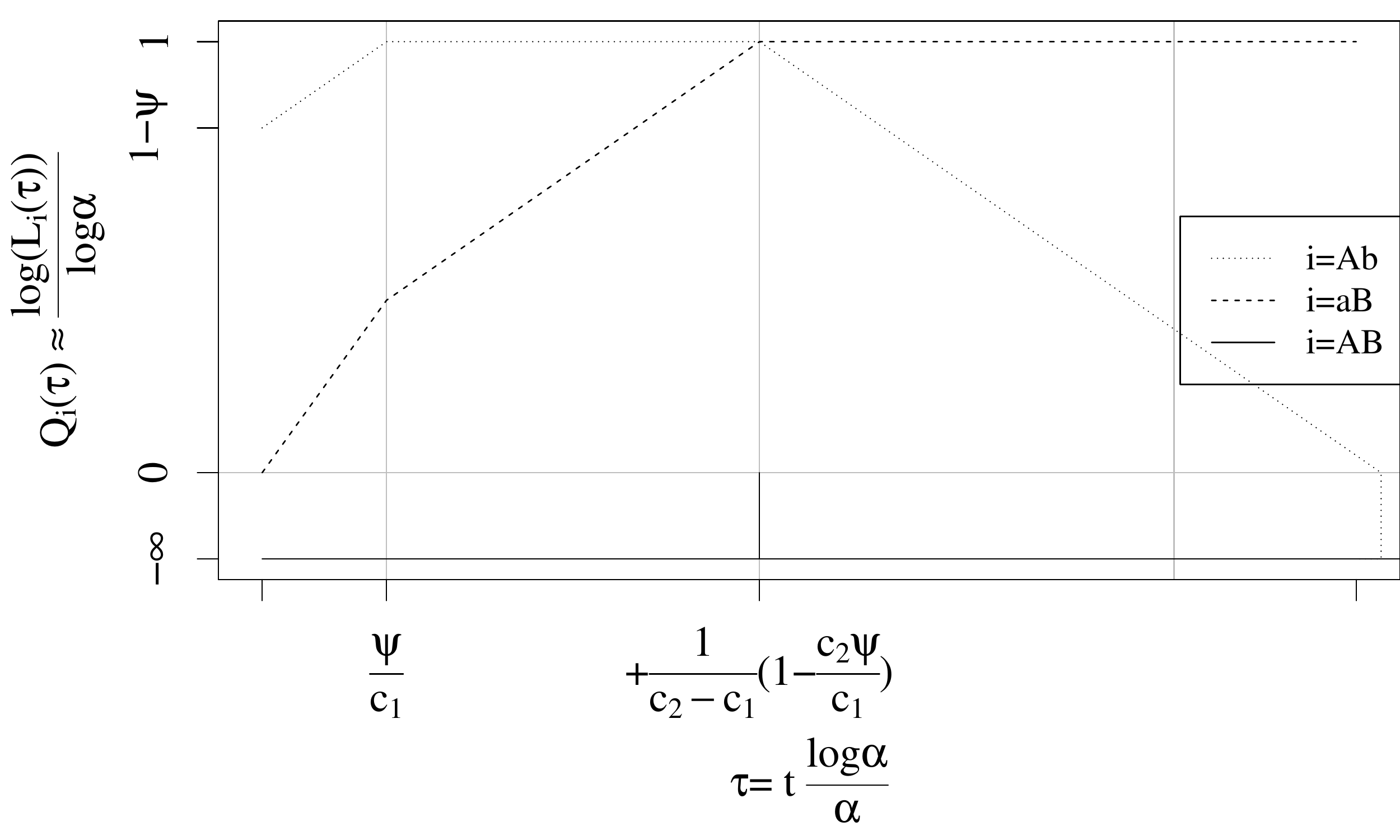}} \\
      (C) \hspace{1cm} & \parbox{9cm}{\includegraphics[width=9cm]{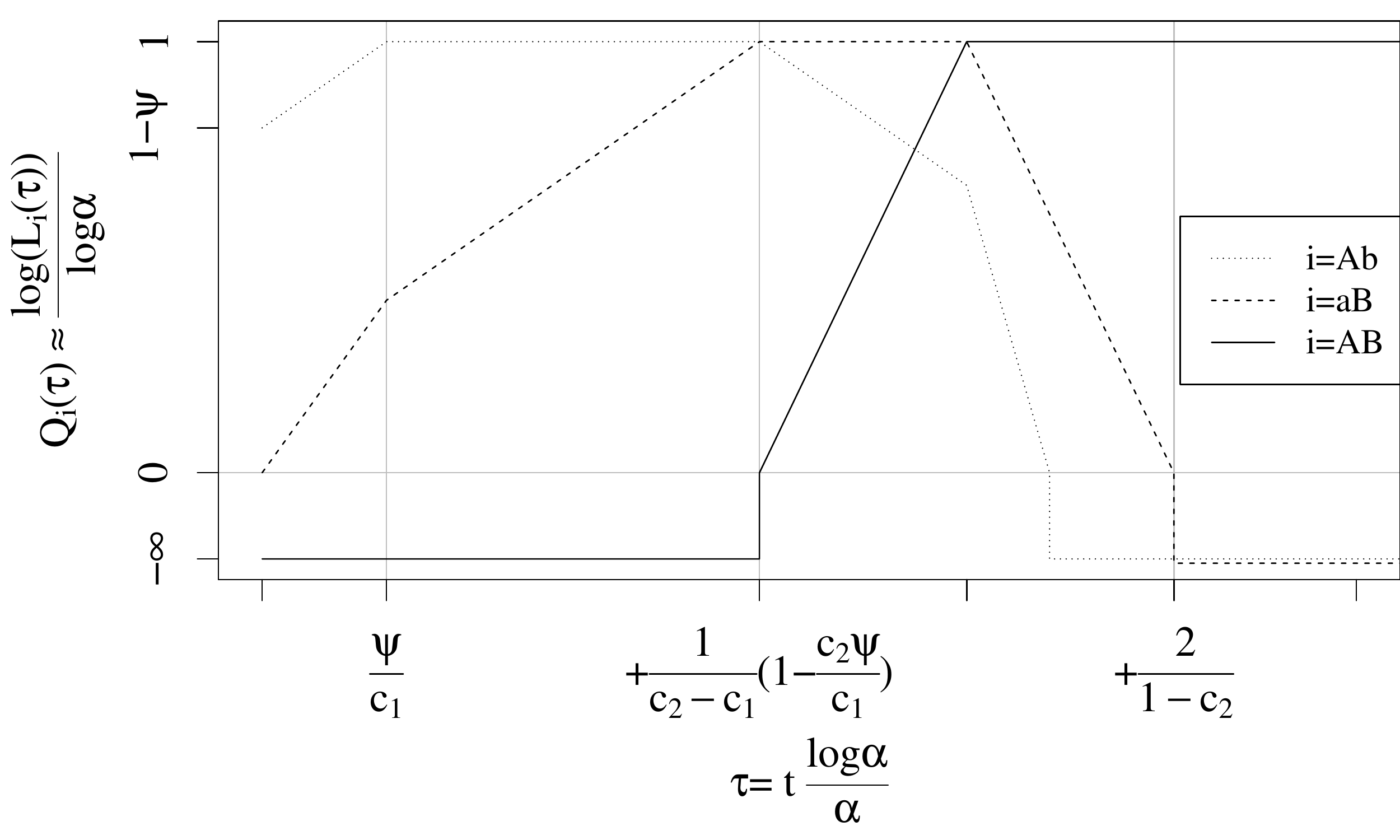}} \\
    \end{tabular}
  \end{center}
  \caption{\label{fig:compSweep1} For the limit of the process
    $(Q^i)_{i=1,2,3}$ as $\alpha\to\infty$ in the case
    $\psi < \frac{c_1}{c_2}$, there are three possibilities. (Note
    that convergence towards $Q^1$ is not claimed in
    Proposition~\ref{P:fixScenario}, but is displayed here for
    completeness.) (A) Type $2\equiv aB$ does not even establish,
    leading to quick fixation of $1\equiv Ab$. No $3\equiv AB$ is
    produced. This happens with probability $1-c_2$. (B) Type $aB$
    establishes, but no successful type $AB$ is created during the
    spread of type $aB$. This happens with probability $c_2 (1-p)$.
    (C) Type $aB$ establishes, successful types $AB$ are created and
    they spread through the whole population. This happens with
    probability $c_2 p$.}
\end{figure}

\begin{figure}
  \hspace{3cm} 
  \begin{center}
    \begin{tabular}{cc}
      (A) \hspace{1cm} & \parbox{9cm}{\includegraphics[width=9cm]{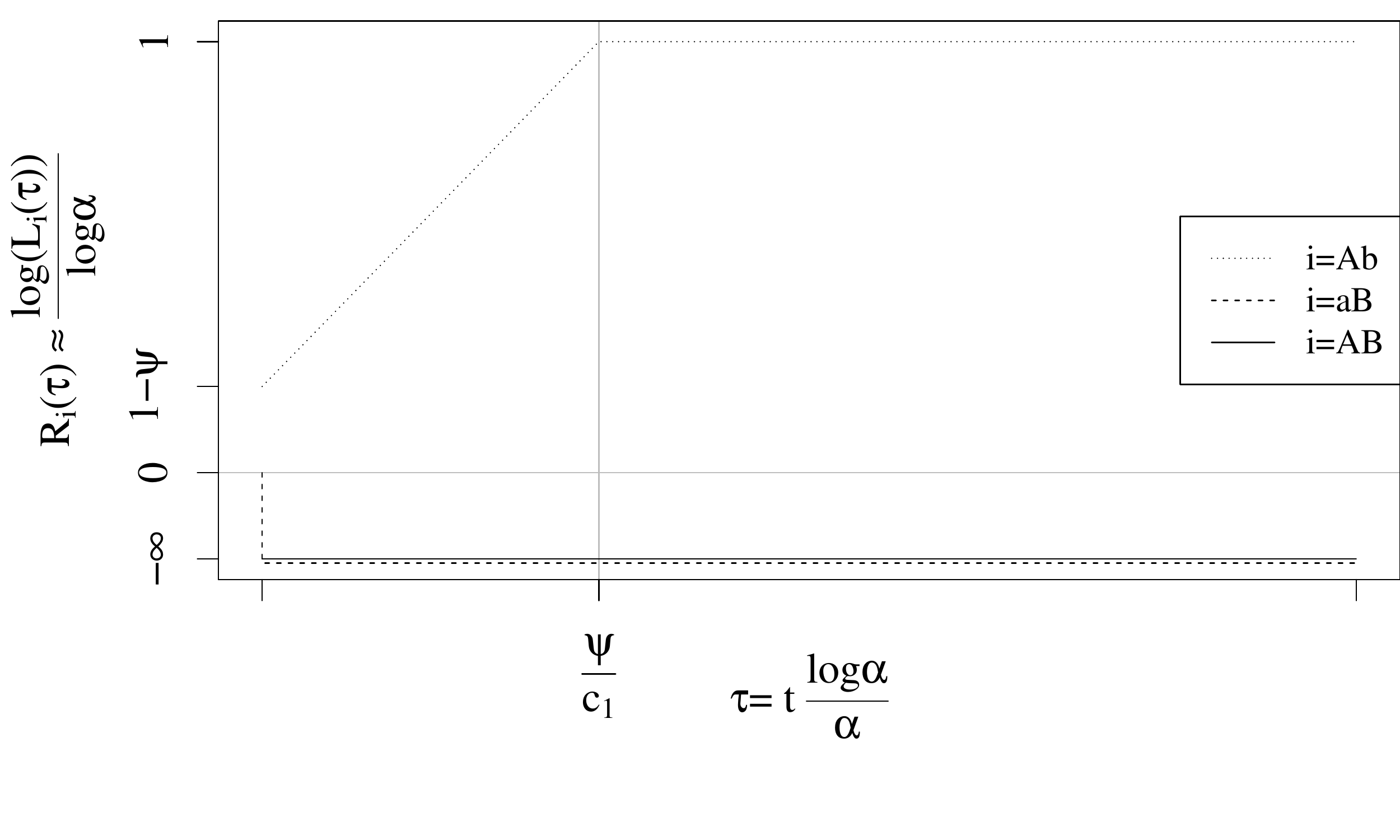}} \\
      (B) \hspace{1cm} & \parbox{9cm}{\includegraphics[width=9cm]{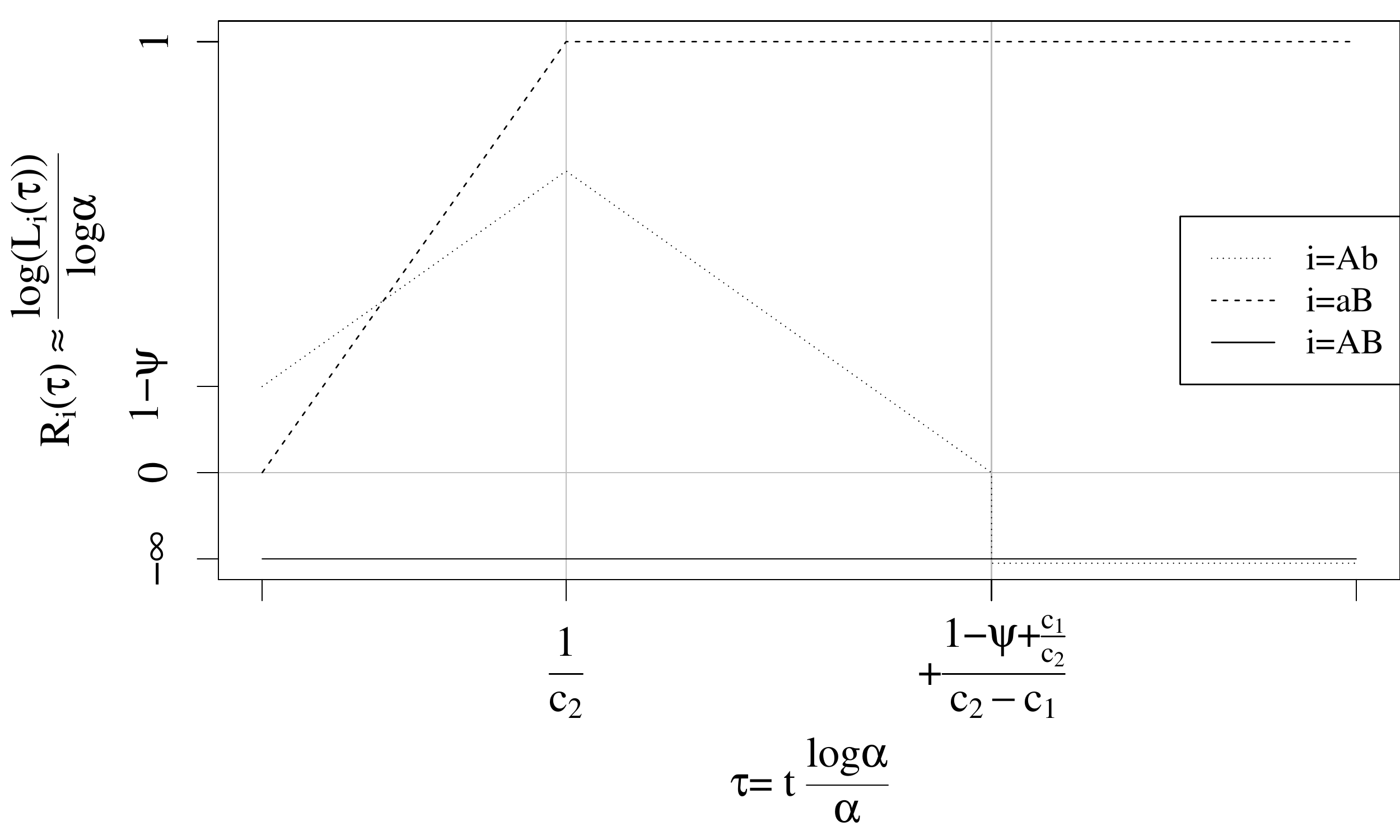}} \\
    \end{tabular}
  \end{center}
  \caption{\label{fig:compSweep2} For the limit of the process
    $(R^i)_{i=1,2,3}$ as $\alpha\to\infty$ in the case
    $\frac{c_1}{c_2} < \psi \leq 1$, there are two possibilities. In
    both cases, type~3 does not occur, so fixation of this type has a
    probability converging to~0.}
\end{figure}

\begin{proof}
  During the proof, we will make use of sequences
  $\varepsilon_\alpha \downarrow 0$. Their precise value will change
  from occurrence to occurence.\\
  We start with the initial phase. Let
  $\ell_1, \ell_2 \leq \varepsilon_\alpha \alpha$ for
  $\varepsilon_\alpha\xrightarrow{\alpha\to\infty}0$. We bound the
  rates $r_i^+, r_i^-, i=1,2$ of in- and decrease of $L_1$ and $L_2$
  before
  $T^1_{\varepsilon_\alpha\alpha}\wedge
  T^2_{\varepsilon_\alpha\alpha}$.
  During this time, we have that
  $\ell_0 \stackrel{\alpha\to\infty} = 2\alpha + o(\alpha)$ by
  Lemma~\ref{l:conc} and $\ell_3=0$ by Lemma~\ref{l:10}. Hence,
  $r_i^+, r_i^-$ satisfy
  \begin{align*}
    r_i^+ & := \alpha\ell_i,\qquad i=1,2,
    \\
    r_1^- & = \frac 12 \ell_1 \ell_0(1 - c_1) + \ell_1 \cdot o(\alpha) = \ell_1 \alpha(1-c_1 + o(1)),
    \\ 
    r_2^- & := \frac 12 \ell_2 \ell_0(1 - c_2) + \ell_2 \cdot o(\alpha) = \ell_2 \alpha(1-c_2 + o(1))
  \end{align*}
  before
  $T^1_{\varepsilon_\alpha\alpha}\wedge
  T^2_{\varepsilon_\alpha\alpha}$.
  So, $L_1$ behaves like a birth-death process as in
  Lemma~\ref{l:growthBin}, starting in $c\alpha^{1-\psi}$ for some
  $c>0, 0<\psi\leq 1$, and $L_2$ behaves like a birth-death process as
  in Corollary~\ref{c3} (starting in $L_2=1$). From the latter,
  $(\log L_2((\tau \log\alpha)/\alpha)/\log\alpha)$ converges to a
  random variable which is $c_2\tau$ with probability $c_2$ and
  $-\infty$ with probability $1-c_2$ as long as
  $(\tau \log\alpha)/\alpha < T_{\varepsilon_\alpha\alpha}^1 \wedge
  T_{\varepsilon_\alpha\alpha}^2$.
  Moreover, according to Lemma~\ref{l:growthBin},
  %$L_1$ hits
  %$z\alpha^p$ at time
  %$\frac{p - (1-\psi)}{c_1} \frac{\log\alpha}{\alpha} +
  %o\Big(\frac{\log\alpha}{\alpha}\Big)$,
  %so 
  $(\log L_1((\tau \log\alpha)/\alpha)/\log\alpha)$ converges to
  $1-\psi + c_1\tau$ as long as
  $(\tau \log\alpha)/\alpha < T_{\varepsilon_\alpha\alpha}^1 \wedge
  T_{\varepsilon_\alpha\alpha}^2$.
  From these conclusions and Lemma~\ref{l:5a}, we see that
  $$ T_{\varepsilon_\alpha\alpha}^1 \wedge
  T_{\varepsilon_\alpha\alpha}^2 = \frac\psi{c_1}\wedge \frac
  1{c_2}.$$
  
  \sloppy 2. We can now proceed to show our result for
  $\frac{c_1}{c_2} < \psi {\color{black} \leq} 1$. With probability
  $c_2$, in the initial phase, according to Corollary~\ref{c3},
  $(\log L_2((\tau \log\alpha)/\alpha)/\log\alpha)$ increases
  approximately linearly with speed $c_2$. In this case, for any
  $\varepsilon_\alpha\downarrow 0$ and
  $\tau_\alpha = \frac 1{c_2} - \varepsilon_\alpha$, we find that
  $(\log L_2((\tau_\alpha) \log\alpha)/\alpha)/\log\alpha)
  \xrightarrow{\alpha\to\infty}1$
  whereas -- from Lemma~\ref{l:growthBin} --
  $(\log L_1((\tau_\alpha \log\alpha)/\alpha)/\log\alpha)
  \xrightarrow{\alpha\to\infty}1 - \psi + \frac{c_1}{c_2}<1$.
  From Lemma~\ref{l:3}, we see that we can choose $\varepsilon_\alpha$
  such that $L_2$ hits $2\alpha(1-\varepsilon_\alpha)$ after some time
  of duration $o((\log\alpha)/\alpha)$. After
  $T^2_{2\alpha(1-\varepsilon_\alpha)}$, we have that
  $L_0 = o(\alpha)$, $L_1 = o(\alpha)$, $L_2 = 2\alpha + o(\alpha)$
  and $L_1$ has rate of decrease
  \begin{align*}
    r_1^- & = \ell_1 \alpha(1-c_1+c_2 + o(1)).
  \end{align*}
  So from here on, $(\log L_1((\tau \log\alpha)/\alpha)/\log\alpha)$
  decreases linearly with speed $c_2-c_1$ due to Lemma \ref{l:5} and
  hits $-\infty$ at time approximately $\sigma_2$. During this whole
  process, the expected number of particles of type~3 which are
  created is bounded for some small $\delta>0$ and some $c>0$ by
  $$ \mathbb E\Big[\int_0^{\sigma_2 \frac{\log\alpha}{\alpha}} \frac{\rho}{\alpha}L_1(s)L_2(s) ds\Big] \leq
  c\frac{\log\alpha}{\alpha} \frac{\rho}{\alpha} \alpha^{1-\psi +
    c_1/c_2 + \delta} \alpha ds = o(1),$$
  so $(\log L_3((\tau \log\alpha)/\alpha)/\log\alpha)=-\infty$ with
  high probability for all $\tau$. This shows all assertions of 2.

  1. We have already seen that initially
  $(\log L_2((\tau \log\alpha)/\alpha)/\log\alpha)$ increases
  approximately linearly with speed $c_2$ with probability $c_2$, and
  with probability $1-c_2$, we have the situation from
  \eqref{P33a}. In the sequel, we assume the linear increase, which
  happens with probability $c_2$.  Since $\psi < \frac{c_1}{c_2}$, we
  find with Lemma~\ref{l:5a} that
  $(\log L_1((\tau_1 \log\alpha)/\alpha)/\log\alpha)
  \xrightarrow{\alpha\to\infty} 1$
  and
  $(\log L_2((\tau_1 \log\alpha)/\alpha)/\log\alpha)
  \xrightarrow{\alpha\to\infty} \psi\frac{c_2}{c_1}$.
  By the fast middle phase of a sweep from Lemma~\ref{l:3}, for some
  $\varepsilon_\alpha\downarrow 0$, it is
  $L_1((\tau_1 + \varepsilon_\alpha)(\log\alpha)/\alpha) = 2\alpha +
  o(\alpha)$. From this, we see that $L_2$ has rate of decrease
  $$ r_2^- = \ell_2\alpha(1-c_2+c_1 + o(1)),$$
  as long as $\ell_2 \leq \varepsilon_\alpha \alpha$ for some
  $\varepsilon_\alpha\downarrow 0$. From Lemma~\ref{l:5}, we see that
  $\log L_2((\tau_1+\tau)(\log\alpha)/\alpha)$ increases linearly at
  speed $c_2-c_1$ until $L_2$ hits $\varepsilon_\alpha \alpha$ for
  some $\varepsilon_\alpha\downarrow 0$, which happens at some time
  $(\tau_2 + \varepsilon_\alpha)(\log\alpha)/\alpha$.  Since we know
  that $T_1^3 > T_{\varepsilon_\alpha \alpha}^2$ with high probability
  for any $\varepsilon_\alpha\downarrow 0$ from Lemma~\ref{l:10}, we
  see from Proposition~\ref{P:rescale} that for $\varepsilon>0$ small
  enough,
  $$\Big(L_1\Big(T^2_{\varepsilon\alpha} + \frac{t}{\alpha}\Big), L_2\Big(T^2_{\varepsilon\alpha} + \frac{t}{\alpha}\Big), 
  L_3\Big(T^2_{\varepsilon\alpha} + \frac t\alpha\Big)\Big)_{t\geq
    0}$$
  converges towards a process $(V^1, V^2, V^3)$ with
  $V^1(0) = 2(1-\varepsilon)$, $V^2(0)=\varepsilon$ and $V^3(0)=0$ as
  in Proposition~\ref{P:rescale}. From Proposition~\ref{P:49} we find
  that $V^3$ survives with probability $p$ and goes extinct with
  probability $1-p$. Therefore, with probability $1-p$, we have
  that $L_3((\tau_2 + \tau)(\log\alpha)/\alpha)=0$ for any $\tau>0$
  (and therefore
  $\log L_3((\tau_2+\tau)(\log\alpha)/\alpha)=-\infty$).  With
  probability $p$, the process $L_3$ survives, so for any small
  $\delta>0$, by time $(\tau_2 + \delta)(\log\alpha)/\alpha$, the
  process $L_3$ has death rate
  $$r_3^- = \ell_3 \alpha (1-c_2 + o(1)),$$ and therefore,
  $\log L_3((\tau_2 + \tau)(\log\alpha)/\alpha)$ increases by
  Lemma~\ref{l:growthBin} approximately linearly at speed $1-c_2$
  until $T^3_{\varepsilon_\alpha\alpha}$ for some
  $\varepsilon_\alpha\downarrow 0$, hence
  $T^3_{\varepsilon_\alpha\alpha}=\tau_3 + o(1)$ with probability
  $p$. From here on, the argument follows along the same line as in
  2.: During a time of duration of order $o((\log\alpha)/\alpha)$,
  $L_3$ grows logistically up to $2(1-\varepsilon_\alpha)\alpha$, and
  then $\log L_2((\tau_3 + \tau)(\log\alpha)/\alpha)$ decreases
  linearly at speed $1-c_2$ until it reaches $0$ and then jumps to
  $-\infty$ at time~$\tau_4$.
\end{proof}

\subsection{Proof of Theorem~\ref{theoremfixprob2}}
\label{Sec:proof1}
1. By Proposition~\ref{P:fixProbL}, the assertion of the Theorem
translates to
\begin{align*}
  \lim_{\alpha\to\infty} \mathbb P(L_j(\infty)=0, j\neq 3) = c_2 p
\end{align*}
with $p$ from {\color{black} case~1 of}
Proposition~\ref{P:fixScenario}. From the latter proposition, we see
that type~3 only fixes (within $L$ in the sense that eventually
$L_j(t)=0$ for $j=0,1,2$) in the case~\eqref{P33c} with probability
$c_2p$, which shows the assertion.\\
2. Here, we have to show that
\begin{align*}
  \lim_{\alpha\to\infty} \mathbb P(L_j(\infty)=0, j\neq 3) = 0.
\end{align*}
{\color{black} First, we treat the case $\psi>1$. Since $L_1(0) \sim \text{Poi}(2\alpha^{1-\psi})$, we find that
  $ L_1 \xRightarrow{\alpha\to\infty} 0$, and therefore 
  $ L_3 \xRightarrow{\alpha\to\infty} 0$, since there is no chance that type $3\equiv AB$ forms due to recombination.
  Therefore, the assertion holds in this case.  In the case
  $c_1/c_2 < \psi \leq 1$, the assertion}
follows from {\color{black}case~2 in}
Proposition~\ref{P:fixScenario}, since
$L_3\xRightarrow{\alpha\to\infty}0$ in all cases.

\subsection{Proof of Theorem~\ref{T:fixTime}}
\label{Sec:proof2}
Recall $\tau_4$ from Proposition~\ref{P:fixScenario} and note that
$$ \tau_4 = \frac{1-\psi}{c_2-c_1} + \frac{2}{1-c_2},$$
which is the limit in probability of $\frac{\alpha}{\log\alpha}S$
from the Theorem. First, for $\varepsilon>0$, using
Proposition~\ref{P:fixProbL}
\begin{align*}
  \lim_{\alpha\to\infty} 
  & 
    \lim_{\delta\to\infty}\mathbb P_{\underline x_{\delta, \psi}}
    \Big(\frac{\alpha}{\log\alpha} S < \tau_4 + \varepsilon \big|S<\infty\Big)
    = \lim_{\alpha\to\infty} \lim_{\delta\to\infty}\frac{\frac{1}{2\alpha\delta}\mathbb 
    P_{\underline x_{\delta, \psi}}
    \Big(\frac{\alpha}{\log\alpha} S < \tau_4 
      + \varepsilon\Big)}{\frac{1}{2\alpha\delta}\mathbb P_{\underline x_{\delta,\psi}}(S<\infty)} 
  \\ & = \lim_{\alpha\to\infty} \frac{\mathbb P\Big(L_j\Big((\tau_4 + \varepsilon)
       \frac{\log\alpha}{\alpha}\Big) = 0, j\neq 3\Big)}{\mathbb P(L_j(\infty) = 0, j\neq 3)}=1
\end{align*}
because, from Proposition~\ref{P:fixScenario}, we see that both the
numerator and denominator equal $c_2p$ in the limit
$\alpha\to\infty$. Moreover, using the same arguments,
\begin{align*}
  \lim_{\alpha\to\infty} 
  \lim_{\delta\to\infty}\mathbb P_{\underline x_{\delta, \psi}}
  \Big(\frac{\alpha}{\log\alpha} S & < \tau_4 - \varepsilon \big|S<\infty\Big)
  \\ & = \lim_{\alpha\to\infty} \frac{\mathbb P\Big(L_j\Big((\tau_4 - \varepsilon)
       \frac{\log\alpha}{\alpha}\Big) = 0, j\neq 3\Big)}{\mathbb P(L_j(\infty) = 0, j\neq 3)} = 0, 
\end{align*}
because the numerator is~0 according to
Proposition~\ref{P:fixScenario}, since no scenario gives fixation of
type~3 already by time $\tau_4 - \varepsilon$.

\subsection*{Acknowledgements}
We thank Martin Hutzenthaler for several discussions on competing
sweeps and Wolfgang Stephan {\color{black} and an anonymous reviewer}
for helpful comments. This research was supported by the DFG through
the research unit 1078 and the priority program 1590, and in
particular through grant Pf-672/3-1 and Pf-672/6-1.

%\bibliographystyle{chicago}
%\bibliography{CompSweeps}

\end{document}